\documentclass{amsart}

\usepackage{amsmath}
\usepackage{amssymb}
\usepackage{amsthm}
\usepackage[english]{babel}
\usepackage{bm}
\usepackage[margin=0.5in]{caption}
\usepackage{enumitem}
\usepackage[T1]{fontenc}
\usepackage{graphicx}
\usepackage[hidelinks]{hyperref} \usepackage[capitalise]{cleveref}
\usepackage{indentfirst}
\usepackage[utf8]{inputenc}
\usepackage{layout}
\usepackage{listings}
\usepackage{lscape}
\usepackage{mathrsfs}
\usepackage{mathtools}
\usepackage{nameref}
\usepackage{setspace}
\usepackage{textcomp}
\usepackage{thmtools}
\usepackage{xcolor}
\usepackage{xparse}

\usepackage{tikz}
\usepackage{tikz-cd}

\title{Spatial cube complexes}
\author{Adrien \textsc{Abgrall}}
\date{\today}

\begin{document}

\pagestyle{plain}
\frenchspacing
\parindent=15pt
\theoremstyle{plain}
\newtheorem{thm}{Theorem}[section]
\newtheorem{lm}[thm]{Lemma}
\newtheorem{pro}[thm]{Proposition}
\newtheorem{cor}[thm]{Corollary}
\newtheorem{conj}[thm]{Conjecture}
\newtheorem*{csvdefi}{Definition}
\newtheorem*{csvlm}{Lemma}
\newtheorem*{csvcor}{Corollary}
\newtheorem*{csvthm}{Theorem}
\newtheorem{innercustomthm}{Theorem}
\newenvironment{customthm}[1]
  {\renewcommand\theinnercustomthm{#1}\innercustomthm}
  {\endinnercustomthm}
\newtheorem{innercustomcor}{Corollary}
\newenvironment{customcor}[1]
  {\renewcommand\theinnercustomcor{#1}\innercustomcor}
  {\endinnercustomcor}
\theoremstyle{definition}
\newtheorem{defi}[thm]{Definition}
\newtheorem{rem}[thm]{Remark}
\theoremstyle{remark}
\newtheorem{q}[thm]{Question}
\newtheorem{ex}[thm]{Example}
\newtheorem{nota}[thm]{Notation}
\newtheorem*{ack}{Acknowledgments}

\newcommand{\NN}{\mathbb{N}}
\newcommand{\ZZ}{\mathbb{Z}}
\newcommand{\QQ}{\mathbb{Q}}
\newcommand{\RR}{\mathbb{R}}
\newcommand{\CC}{\mathbb{C}}
\renewcommand{\SS}{\mathbb{S}}
\newcommand{\TT}{\mathbb{T}}
\newcommand{\HH}{\mathbb{H}}
\newcommand{\FF}{\mathbb{F}}

\newcommand{\cat}{\mathrm{CAT}(0)}
\newcommand{\s}{\mathrm{Sep}}
\newcommand{\out}{\mathsf{Out}}
\newcommand{\aut}{\mathsf{Aut}}
\newcommand{\inn}{\mathsf{Inn}}
\newcommand{\unt}{\mathcal{U}}
\newcommand{\uaut}{\mathcal{U}\mathsf{Aut}}
\renewcommand{\H}{\mathfrak{H}}
\newcommand{\K}{\mathfrak{K}}
\newcommand{\W}{\mathfrak{W}}
\renewcommand{\emptyset}{\varnothing}

\newcommand{\n}[1]{\left|\left|\,#1\,\right|\right|}
\newcommand{\floor}[1]{\left\lfloor #1 \right\rfloor}
\newcommand{\ceil}[1]{\left\lceil #1 \right\rceil}
\newcommand{\gen}[1]{\left\langle #1 \right\rangle}
\newcommand{\card}[1]{\left | #1 \right |}
\newcommand{\gl}{GL}
\newcommand{\uc}[1]{\widetilde{#1}}
\newcommand{\USS}{\uc{\SS}}
\newcommand{\sd}{\mathbin{\triangle}}

\newcommand{\bin}[1]{}
\newcommand{\todo}[1]{{\color{red}#1}}

\begin{abstract}
We provide a new geometric characterization for the spine of untwisted outer space of a right-angled Artin group, constructed by Charney, Stambaugh, and Vogtmann. We realize the spine as the natural simplicial complex associated with the category of \emph{spatial cube complexes}, a new class of locally \texorpdfstring{$\cat$}{CAT(0)} cube complexes with a simple definition in terms of hyperplane collapses.
\end{abstract}

\maketitle

\section{Introduction}

$\cat$ cube complexes have been introduced by Gromov (\cite{gromov}) as a combinatorial model for non-positive curvature. Since then, the study of groups acting by isometries on $\cat$ cube complexes, or discrete wallspaces (\cite{ccc}, \cite{wallspace}) has become an important topic in geometric group theory. Notably, the work of Haglund and Wise on \emph{special cube complexes} (\cite{special}) has had far-reaching consequences, emphasizing the strong links between $\cat$ cube complexes and right-angled Artin groups. These groups form a natural class of finitely presented groups stable under free product and direct product, containing both finite-rank free and free abelian groups. They are the fundamental groups of a class of special cube complexes called the \emph{Salvetti complexes}.

One of the first known results about automorphism groups of right-angled Artin groups is due to Laurence (\cite{laurence}) and Servatius (\cite{servatius}). They found a finite generating set for the group $\out(A)$ of outer automorphisms of any right-angled Artin group $A$, closely resembling the Nielsen generators of $\out(F_n)$, the outer automorphism group of a free group. This generating set can contain \emph{twists}: automorphisms related to transvections in $GL_n(\ZZ) = \out(\ZZ^n)$. However, the \emph{untwisted subgroup} $\unt(A)\leq \out(A)$, the subgroup generated by all the Laurence-Servatius generators except the twists, appears very similar to $\out(F_n)$. It is worth noting that in many right-angled Artin groups, such as free groups themselves, $\unt(A) = \out(A)$, and that, in general, the untwisted subgroup was characterized by Fioravanti (\cite{coarsemedian}) as the group of automorphisms preserving coarsely a natural median structure on $A$.

Drawing on this similarity, and the definition due to Culler and Vogtmann (\cite{cullervogtmann}) of the outer space for free groups, Charney, Stambaugh, and Vogtmann (\cite{untwistedos}) defined a \emph{spine of untwisted outer space} for any right-angled Artin group $A$. This spine is a contractible finite-dimensional simplicial complex $K$ endowed with a properly discontinuous and cocompact action of $\unt(A)$ by isometries. It provides an interesting geometric model for the group and allows for explicit computation of homological invariants. The quotient of $K$ by a torsion-free subgroup of $\unt(A)$ of finite index is a classifying space for this subgroup, recovering the fact that $\unt(A)$ is finitely presented (which was originally due to Day in \cite{daypresentation}, where $\unt(A)$ is designated as the subgroup generated by \emph{long-range automorphisms}).

Like for the Culler-Vogtmann spine, $K$ is defined as a set of actions of $A$ on a certain class of $\cat$ cube complexes, and $\unt(A)$ acts by precomposition with the action. This particular class of $\cat$ cube complexes is described in \cite{untwistedos} as obtained from the universal cover of the Salvetti complex by a \emph{blow-up} procedure. The definition is entirely explicit, but highly technical.

\bigskip

In this paper, we provide a new description of the Charney-Stambaugh-Vogtmann spine of untwisted outer space $K$. This description is based on combinatorial properties of hyperplane collapses in $\cat$ cube complexes, which appear naturally from the viewpoint of groups acting on discrete wallspaces.
\begin{defi}[see Definitions~\ref{defcollapse} and \ref{collapses}]
    Let $G$ be act by combinatorial isometries on a $\cat$ cube complex $Z$, inducing an action of $G$ on the set $\W$ of hyperplanes of $Z$. The \emph{collapse} of a $G$-invariant family $\H\subseteq\W$ of hyperplanes of $Z$ is a $G$-equivariant cellular map of $\cat$ cube complexes with domain $Z$ obtained by collapsing to a point each edge dual to a hyperplane of $\H$. This has a simple interpretation in terms of wallspace maps, corresponding to Caprace and Sageev's \emph{restriction quotients} (\cite{rankrigidity}, see Definition~\ref{restriction} and Lemma~\ref{collapse1}).
    
    When $G$ acts without cube inversions, a \emph{strong collapse} is the collapse $c\colon Z\to Y$ of a $G$-invariant family of hyperplanes $\H$ satisfying any of the following equivalent conditions: 
    \begin{enumerate}
        \item For every $H\in \H$, and every $g\in G\setminus Stab(H)$, some hyperplane separating $H$ from $gH$ is not contained in $\H$
        \item For every $H\in \H$ with carrier $K$, the subcomplex $c(K)\subseteq Y$ is disjoint from its distinct $G$-translates
        \item For every $H\in \H$ with carrier $K$, the immersion $c(K)/Stab(H)\to Y/G$ is an embedding.
    \end{enumerate}
    When $G$ acts freely on $Z$ and $Y$, we describe a further equivalent condition in the quotient as follows: Parallel edges of $Z/G$ lying inside the same vertex preimage of the quotient map $Z/G\to Y/G$ are parallel in that vertex preimage (Lemma~\ref{strongdownstairs}).
\end{defi}

This new property of collapses, resembling part of the definition of special cube complexes, appears to be the right tool to characterize the Charney-Stambaugh-Vogtmann blow-ups geometrically as \emph{spatial cube complexes}, defined as follows.
\begin{defi}[see Definitions~\ref{defspatial}, \ref{defuntwisted}, \ref{defcat}]
    Let $A$ be a right-angled Artin group acting freely and cocompactly on a $\cat$ cube complex $Z$. Let $\SS$ be the associated Salvetti complex, with a standard action of $A$ on its universal cover $\uc{\SS}$. The action of $A$ on $Z$ is:
    \begin{itemize}
        \item \emph{Salvetti-like} when the quotient $Z/A$ is isomorphic to $\SS$;
        \item \emph{cospatial} when $Z/A$ cannot be "unsubdivided" (in a sense made precise in Definition~\ref{redundantdef}), and for every hyperplane $H$ of $Z$, there exists a strong collapse with domain $Z$, with Salvetti-like range, and that does not collapse $H$;
        \item \emph{untwisted}, or \emph{coarse-median-preserving}, when there exists vertices $x\in \uc{\SS}$, $z\in Z$ and a constant $N$ such that for $a,b,c\in A$, the unique $g\in A$ such that $gx = \mu_{\uc{\SS}}(ax,bx,cx)$ satisfies $d_Z(gz, \mu_Z(az,bz,cz))\leq N$, where $\mu_{\uc{\SS}}$ and $\mu_Z$ are the medians on the $\cat$ cube complexes $\uc{\SS}$ and $Z$ respectively.
    \end{itemize}
    A \emph{spatial cube complex} is the quotient of a cospatial action of $A$.
\medskip

    Let $\overline{\mathcal{C}}$ be the category with cospatial untwisted actions of $A$ as objects and strong collapses as arrows, considered modulo $A$-equivariant combinatorial isomorphism. This category is naturally associated to a simplicial complex $\mathcal{N}$ (its nerve), where $0$-cells correspond to objects, $1$-cells to arrows, $k$-cells to sequences of $k$ composable arrows.
\end{defi}

We are then able to prove the following reinterpretation.

\begin{customthm}{\ref{newspineth}}
    The group $\unt(A_\Gamma)$ acts on untwisted cospatial actions by precomposition. It acts on $\mathcal{N}$ by combinatorial isomorphisms.

    There exists an $\unt(A_\Gamma)$-equivariant, order-preserving, combinatorial isomorphism between the spine of untwisted outer space $K_\Gamma$ and $\mathcal{N}$.
\end{customthm}

We also prove that every hyperplane collapse $c$ between spatial cube complexes for $A$ is the quotient of a strong collapse (Corollary~\ref{strong}). Hence, any equivariant hyperplane collapse between objects of $\overline{\mathcal{C}}$ defines an arrow of $\overline{\mathcal{C}}$.

\bigskip
To briefly outline the contents of this paper, in Section 2 we recall useful properties of $\cat$ cube complexes and give a short account of the construction of the Charney-Stambaugh-Vogtmann spine. In Section 3, we define and relate several properties of hyperplane collapses in $\cat$ cube complexes with a group action. In Section 4, we investigate further properties of such collapses when the domain or the range is chosen among the Charney-Stambaugh-Vogtmann spine. Finally in Section 5, we use previously defined strong collapses to define the category of spatial cube complexes and prove the main theorem.

\begin{ack}
I gratefully acknowledge support from project ANR-22-CE40-0004 GoFR. Thank you Sam Fisher and Zachary Munro for our inspiring conversations. I am very thankful to my advisor, Vincent Guirardel, for our many enlightening discussions, his patience and his continual support.
\end{ack}

\section{Background}

\subsection{\texorpdfstring{$\cat$}{CAT(0)} cube complexes}

\emph{$\cat$ cube complexes} are famous examples of metric cell complexes coming from work of Gromov (\cite{gromov}) and Sageev (\cite{ccc}). They have many equivalent characterizations as \emph{median graphs} (\cite{median}) or \emph{discrete wallspaces} (\cite{wallspace}). We only recall some of their main features and refer the reader to \cite{richestoraags}, \cite{genevoisbook} for a more thorough approach.

\begin{defi}
    A \emph{cube complex} is a cell complex obtained by gluing together Euclidean cubes (isometric to $[0,1]^n$ for various integers $n$) isometrically along subcubes. It is endowed with a piecewise Euclidean length metric extending the standard metric on each cube. A \emph{$\cat$ cube complex} (resp. \emph{locally $\cat$ cube complex}) is a cube complex where the piecewise Euclidean metric is $\cat$ (resp. \emph{locally $\cat$}). Being locally $\cat$ is equivalent to satisfying \emph{Gromov's link condition}, a combinatorial condition on vertex links (see \cite{bridsonhaefliger}, Chapter II, Theorem~5.5). The $0$-skeleton of a $\cat$ cube complex $X$ is endowed with a different metric, induced by the graph metric on the $1$-skeleton $X^{(1)}$. It is this \emph{combinatorial metric} that we will consider, denoted $d$.
    
    When $X$ is $\cat$, for each triple of vertices $(x_1,x_2,x_3)\in X^{(0)}$ there exists a unique \emph{median} $\mu(x_1,x_2,x_3)\in X^{(0)}$ with the following property: for $i\neq j$, $d(x_i,x_j) = d(x_1,\mu(x_1,x_2,x_3)) + d(\mu(x_1,x_2,x_3), x_j)$.

    When $X$ is locally $\cat$, say two edges $e,f$ of $X$ are \emph{parallel} when there exists a sequence of edges $e=e_0,\dots,e_n=f$ such that for $1\leq i\leq n$, $e_{i-1}$ and $e_i$ are opposite edges in a square of $X$. This is an equivalence relation on edges whose equivalence classes are called the \emph{hyperplanes} of $X$. An edge $e$ and a hyperplane $H$ are \emph{dual} if $e\in H$. The \emph{carrier} (resp. \emph{open carrier}) of $H$ is the union of closed cells (resp. open cells) of $X$ whose closure contains an edge dual to $H$. It is a connected subcomplex (resp. subspace) of $X$. Two hyperplanes of $X$ are \emph{transverse} if their open carriers intersect, or equivalently if they have dual edges that span a square in $X$.

    When $X$ is $\cat$, for every hyperplane $H$ with open carrier $C$, $X\setminus C$ has two connected components called the \emph{halfspaces of $H$}, which are subcomplexes of $X$. Hyperplanes are transverse if and only if their halfspaces pairwise intersect, i.e.~removing both open carriers separates $X$ into four connected components. We will say a quadruple of vertices of $X$ witnesses the transversality of the two hyperplanes if it contains one point per such component. For $A,B$ non-empty subspaces of $X$, the \emph{separator} $\s(A\mid B)$ is the set of hyperplanes $H$ such that $A$ is contained in one halfspaces of $H$ and $B$ is contained in the other. We will sometimes use this notation with $A$ or $B$ being a hyperplane instead of a subspace. The definition remains the same taking as a subspace the carrier of the hyperplane.
\end{defi}

\begin{defi}[Sageev's construction, restriction quotients]
\label{restriction}
    A \emph{discrete wallspace} is a set $E$ together with a set $\mathcal{W}$ of partitions of $E$ in two parts, such that two points of $E$ are always in the same part of all but finitely many partitions in $\mathcal{W}$. When $Z$ is a $\cat$ cube complex, $Z^{(0)}$ with the set of partitions given by halfspaces of hyperplanes is a discrete wallspace. Conversely for any wallspace $(E,\mathcal{W})$ the following construction defines a $\cat$ cube complex $X(E,\mathcal{W})$:
    \begin{itemize}
        \item Each vertex of $X(E,\mathcal{W})$ is the choice of one part in each partition such that the chosen parts pairwise intersect and every vertex of $E$ belongs to all but finitely many of the chosen parts.
        \item Two distinct vertices of $X(E,\mathcal{W})$ are linked by an edge if and only if they have the same chosen parts for all partitions except one.
        \item $2^k$ distinct vertices of $X(E,\mathcal{W})$ are the vertices of a $k$-cube if and only if they have the same chosen parts for all partitions except $k$. Each possible choice of parts for these $k$ partitions corresponds to one of the vertices.
    \end{itemize}
    The set of hyperplanes of $X(E,\mathcal{W})$ canonically identifies with $\mathcal{W}$. Starting with a $\cat$ cube complex $Z$ with set of hyperplanes $\W$, the $\cat$ cube complex $X(Z,\W)$ canonically identifies with $Z$.

    Let $Z$ be a $\cat$ cube complex with set of hyperplanes $\W$ and let $\W'\subseteq \W$. Then $(Z^{(0)},\W')$ is a discrete wallspace and the map $id_{Z^{(0)}}\colon (Z^{(0)},\W)\to (Z^{(0)},\W')$ induces a cellular map $Z\simeq X(Z^{(0)},\W)\to X(Z^{(0)},\W')$ of $\cat$ cube complexes, called the \emph{restriction quotient of $Z$ associated with $\W'$}. The map is defined on vertices by mapping a choice of parts made for all partitions in $\W$ to the same choice of parts for the subset of partitions $\W'$. Adjacent vertices in $Z$ are mapped to equal or adjacent vertices in $X(Z^{(0)}, \W')$, and the map extends affinely to $k$-cubes.
\end{defi}
Sageev's construction was introduced in \cite{ccc}, and restriction quotients were introduced by Caprace and Sageev in \cite{rankrigidity}. The following lemmas highlight the relationship between hyperplanes and edge paths in $\cat$ cube complexes.

\begin{lm}
\label{square}
In a $\cat$ cube complex, two edges dual to the same hyperplane and incident at a vertex are equal. Two edges dual to transverse hyperplanes and incident at a vertex span a square at that vertex. Moreover, every embedded edge cycle is of even length at least $4$, and every embedded edge cycle of length $4$ bounds a square.
\end{lm}
The first part of this result is a consequence of Lemma~3.3 and Lemma~3.6 in \cite{richestoraags}. Any edge cycle crosses every hyperplane an even number of times, hence is of even length. Thus, there are no embeddeded edge cycles of length $2$, and in every embedded cycle of length $4$, pairs opposite edges belong to two distinct parallelism classes. The four vertices of the cycle witness the transversality of the two hyperplanes involved, and the first part of the lemma asserts that two consecutive edges in this cycle span a square. The two remaining edges of the cycle are parallel to the two remaining sides of the square and incident to them, thus equal.

A straightforward induction gives the following corollary
\begin{cor}
\label{product}
Let $H$ be a hyperplane in a $\cat$ cube complex $X$ and $S$ a subcomplex of the carrier of $H$, contained in one halfspace of $H$. Then $S$ spans a product with an edge dual to $H$. In other words, there exists a cellular map $[0,1]\times S\to X$ injective on open cells mapping $\{0\}\times S$ to $S$ and $[0,1]\times \{v\}$ to an edge dual to $H$ for each vertex $v$ of $S$.
\end{cor}

\begin{lm}
\label{edgepath}
In a $\cat$ cube complex, an edge path is geodesic if and only if it crosses no hyperplane twice. In particular, the distance between two vertices is equal to the cardinality of their separator.
\end{lm}
The first part of this result is Lemma~3.9 in \cite{richestoraags}. It implies that a geodesic path crosses a hyperplane if and only if its endpoints belong to distinct halfspaces of the hyperplane. The second part follows.

\begin{defi}[Convexity, local convexity]
    When $X$ is $\cat$, a subcomplex $Y\subseteq X$ is \emph{convex} if it is full (every cube of $X$ with all vertices in $Y$ is in $Y$) and geodesically convex (every geodesic edge path in $X$ between vertices of $Y$ is contained in $Y$). Geodesic convexity is equivalent to the inclusion $\mu(Y^{(0)},Y^{(0)},X^{(0)})\subseteq Y^{(0)}$. When $Y$ is convex, $Y$ is itself $\cat$ and the inclusion map is an isometric embedding. Carriers and halfspaces of hyperplanes in $X$ are convex. A subcomplex is convex if and only if it is an intersection of halfspaces (the empty intersection corresponding to $X$ itself).

    Likewise, when $X$ is locally $\cat$ a subcomplex $Y\subseteq X$ is \emph{locally convex} if all the connected components of the preimage of $Y$ in the universal cover of $X$ are convex. When $Y$ is locally convex, $Y$ is itself locally $\cat$ and the inclusion map is locally isometric and $\pi_1$-injective.
\end{defi}

\begin{lm}[Helly property for convex subcomplexes]
In a $\cat$ cube complex, if finitely many convex subcomplexes pairwise intersect, they intersect globally.
\end{lm}
\begin{proof}
The proof is by induction on the number of subcomplexes. If $A_1, A_2, A_3$ pairwise intersect, let $\alpha,\beta,\gamma$ belong to each of the three intersections. The median $\mu(\alpha,\beta,\gamma)$ belongs to $A_1$ because two of the points do and $A_1$ is convex. Symmetrically it belongs to $A_2$ and $A_3$, proving the global intersection is non-empty.

If $A_1, A_2, \dots A_n$ pairwise intersect, they intersect triplewise by the previous case, thus $A_1\cap A_n, A_2\cap A_n, \dots, A_{n-1}\cap A_n$ are pairwise intersecting convex subcomplexes. Applying the induction hypothesis, the $A_i$ intersect globally.
\end{proof}

\begin{defi}[Special cube complex, see \cite{special}, \emph{A-special}, Definition~3.2 or \cite{richestoraags}, Definition~4.2]
\label{special}
Let $G$ be a group acting freely on a $\cat$ cube complex $X$ by isometries. This action is \emph{cospecial} when it satisfies the following assumptions.
\begin{itemize}
    \item (2-sidedness) If $g\in G$ stabilizes a hyperplane of $X$, $g$ does not exchange its halfspaces.
    \item (no self-intersection) For $g\in G$ and $H$ a hyperplane of $X$, $H$ and $gH$ are not transverse.
    \item (no direct self-osculation) For $g\in G$ and $H$ a hyperplane of $X$ with halfspaces $H^+, H^-$, if the carriers of $H$ and $gH$ intersect, then $H^+$ and $gH^+$ intersect, and $H^-$ and $gH^-$ intersect.
    \item (no inter-osculation) For $g\in G$ and $H,K$ transverse hyperplanes of $X$, if the carriers of $gH$ and $K$ intersect, then $gH$ and $K$ are transverse.
\end{itemize}
A \emph{special cube complex} is the quotient of a cospecial action. It is a locally $\cat$ cube complex.
\end{defi}

\subsection{Untwisted outer space for right-angled Artin groups}

Right-angled Artin groups are well-known example of finitely presented groups. We only recall some of their properties, a more detailed account can be found in \cite{introraags}.

\begin{defi}
    Let $\Gamma = (V,E)$ be a non-empty finite simple graph. The corresponding \emph{right-angled Artin group} is the group $A_\Gamma$ given by the following presentation.
    \[A_\Gamma=\gen{V\mid [v_i,v_j]\; \forall \{v_i,v_j\}\in E}\]

    \emph{Standard generators} of $A_\Gamma$ are the elements of $V$. Let $V^\pm = V\sqcup V^{-1}\subset A_\Gamma$. For $x\in V^\pm$ the \emph{link of $x$} is the set $lk(x)$ defined by the intersection of $V^\pm\setminus\{x,x^{-1}\}$ with the centralizer of $x$. When $x\in V$, elements of $lk(x)$ are the vertices of $\Gamma$ adjacent to $x$ and their inverses.

    The \emph{Salvetti complex} $\SS$ associated with $A_\Gamma$ is the subcomplex of the torus $(S^1)^V$ whose cubes correspond one-to-one to the cliques of $\Gamma$. It is, up to isomorphism, the only special cube complex with one vertex with fundamental group isomorphic to $A_\Gamma$. Hyperplanes of $\SS$ have a canonical bijective labeling by $V$. Two hyperplanes are transverse if and only if their labels are adjacent in $\Gamma$.
\end{defi}

In this paper, we are interested in the outer automorphism group $\out(A_\Gamma)$ and more precisely to its \emph{untwisted subgroup} $\unt(A_\Gamma)\leq \out(A_\Gamma)$. A definition can be found in \cite{untwistedos}, Section~2.2, but we quote the following characterization due to Fioravanti.

\begin{thm}[\cite{coarsemedian}, Proposition~A]
\label{CMP}
    An outer automorphism $[\varphi]\in \out(A_\Gamma)$ is untwisted if and only if for some (any) representative $\varphi \in \aut(A_\Gamma)$, the standard action of $A_\Gamma$ on the universal cover $\uc{\SS}$, denoted $(g,x)\mapsto g\cdot x$, and the action given by $(g,x)\mapsto \varphi(g)\cdot x$ induce the same coarse median structure on $A_\Gamma$ \textit{via} pulling back the median on the universal cover $\uc{\SS}$.
\end{thm}

The group $\unt(A_\Gamma)$ acts on the \emph{spine of untwisted outer space}, $K_\Gamma$, which was constructed in \cite{untwistedos}, and of which we will provide another construction. Below, we recall the main features of the construction in \cite{untwistedos}.

\begin{defi}[Whitehead partition]
Consider partitions $\mathbf{P} = (\{P,P^*\},L)$ of $V^\pm$ into three parts: $V^\pm = P\sqcup P^*\sqcup L$, with $P$ and $P^*$ non-empty. One distinguished part, $L$, is called the \emph{link} of the partition, and the other parts $P,P*$ are the \emph{sides} of the partition. We identify each partition $(P,P^*,L)$ with its \emph{opposite} $(P^*,P,L)$. The following notation will be useful:
\[\begin{aligned}
lk(\mathbf{P}) &= L\\
single(\mathbf{P}) &= \{x\in V^\pm\mid x\in P,\,x^{-1}\in P^* \text{ or }x\in P^*,\,x^{-1}\in P\}\\
double(P) &= \{x\in V^\pm \mid x,x^{-1}\in P\}\\
double(P^*) &= \{x\in V^\pm \mid x,x^{-1}\in P^*\}
\end{aligned}\]

A \emph{based Whitehead partition} $(\mathbf{P},b)$ is a partition $\mathbf{P}$ of the form above, together with a \emph{basepoint} $b\in V^\pm$, satisfying the following hypotheses:
\begin{itemize}
\item $b\in P$, $b^{-1}\in P^*$, and $L=lk(b)$.
\item If $x\in P$ and $x^{-1}\in P^*$, $lk(x)\subseteq L$ (we say that $\mathbf{P}$ \emph{splits} $x$).
\item If $x\in P$ and $y\in P^*$ are not inverses, $x$ and $y$ do not commute.
\item $P$ and $P^*$ have at least two elements each.
\end{itemize}
A \emph{Whitehead partition} is a partition $\mathbf{P}$ of the form above such that there exists $b\in V^\pm$ making $(\mathbf{P},b)$ a based Whitehead partition. Such a $b$ is not unique in general. A Whitehead partition is entirely determined by any basepoint and one side.

A Whitehead partition $\mathbf{P}$ is \emph{adjacent} to a standard generator $v$ when some (any) basepoint of $\mathbf{P}$ is adjacent to $v$ in $\Gamma$. Likewise, two Whitehead partitions $\mathbf{P}$ and $\mathbf{Q}$ are \emph{adjacent} when some (any) basepoints for $\mathbf{P}$ and $\mathbf{Q}$ are adjacent in $\Gamma$. The partitions $\mathbf{P}$ and $\mathbf{Q}$ are \emph{compatible} if they are adjacent or if exactly one of $P\cap Q, P\cap Q^*, P^*\cap Q, P^*\cap Q^*$ is empty (see \cite{twistedos}, Definitions~2.7 and 2.8, amending \cite{untwistedos}, Definition~3.3).
\end{defi}

\begin{defi}[Spine of untwisted outer space]
\label{defspine}
Let $X$ be any locally $\cat$ cube complex with fundamental group isomorphic to $A_\Gamma$. A \emph{marking} on $X$ is a class of group isomorphisms $m\colon \pi_1X\to A_\Gamma=\pi_1\SS$, modulo conjugacy on both sides (allowing us to omit basepoints). A map $f\colon (X,m)\to (X',m')$ between locally $\cat$ cube complexes with markings is said to \emph{preserve the markings} if $m'\circ f_* = m$ (still modulo conjugacy on both sides).

Given $\mathbf{\Pi}$ a collection of pairwise compatible Whitehead partitions, one can construct a \emph{blow-up} $\SS^\mathbf{\Pi}$ of the Salvetti complex $\SS$ (see \cite{untwistedos}, Section~3). It is a locally $\cat$ cube complex with a cellular homotopy equivalence $c_\mathbf{\Pi}\colon \SS^\mathbf{\Pi}\to \SS$ called \emph{canonical collapse}. The hyperplanes of $\SS^\mathbf{\Pi}$ are oriented and labelled bijectively by $V\sqcup \mathbf{\Pi}$ (we will often use "hyperplanes" and "labels" indiscriminately in this context, and may speak of the label of an edge to mean the label of its dual hyperplane). A marking $m$ on $\SS^\mathbf{\Pi}$ is \emph{untwisted} if $[m\circ (c_\mathbf{\Pi})_*^{-1}]\in \unt(A_\Gamma)$. This definition does not depend on the choice of $c_\mathbf{\Pi}$: if $\iota \colon \SS^\mathbf{\Pi}\to \SS^{\mathbf{\Pi}'}$ is an isomorphism, $[(c_{\mathbf{\Pi}'})_*\circ \iota_*\circ (c_\mathbf{\Pi})_*^{-1}]\in \unt(A_\Gamma)$ (see \cite{untwistedos}, Corollary~4.13). The group $\unt(A_\Gamma)$ acts on untwisted markings of blow-ups of $\SS$ by composition on the left. The set of blow-ups of $\SS$ with untwisted markings considered up to marking-preserving combinatorial isometry, endowed with this action of $\unt(A_\Gamma)$, is the $0$-skeleton of $K_\Gamma$.
\bigskip

A hyperplane $H$ in a blow-up $\SS^{\mathbf{\Pi}}$ is a \emph{carrier retract} when the carrier of $H$ embeds as a product $[0,1]\times Y_H$. In that case, the coordinate projection $[0,1]\times Y_H\to Y_H$ extends to a \emph{collapse map} $c_H\colon \SS^\mathbf{\Pi}\to \SS^\mathbf{\Pi}_H$, which is a homotopy equivalence from $\SS^{\mathbf{\Pi}}$ to a new locally $\cat$ cube complex. More generally, this can be done with any family $\H$ of carrier retract hyperplanes of $\SS^\mathbf{\Pi}$, provided that every edge cycle whose edges are only dual to elements of $\H$ is nullhomotopic, to ensure that the collapse map remains a homotopy equivalence.

\begin{thm}(\cite{untwistedos}, Theorem~4.11)
\label{treelikecompatible}
The following assertions are equivalent for $\H$ a family of carrier retract hyperplanes in a blow-up $\SS^\mathbf{\Pi}$.
\begin{enumerate}
    \item The collapse of $\H$, $c_\H\colon \SS^\mathbf{\Pi}\to \SS^{\mathbf{\Pi}}_\H$ is a homotopy equivalence with range combinatorially isomorphic to $\SS$.
    \item For every maximal subset $\mathbf{\Pi}'\subseteq \mathbf{\Pi}$ of partitions all having the same link $L$, the set of edges labelled by $\H$ in $\SS^{\mathbf{\Pi}'}$ forms a maximal subtree of $\Theta\subseteq \SS^{\mathbf{\Pi}'}$, the graph formed by the edges whose label has link $L$.
\end{enumerate}
\end{thm}
A \emph{tree-like set of hyperplanes} is a family $\H$ of hyperplanes in $\SS^\mathbf{\Pi}$ satisfying the two equivalent properties above. Any subfamily of $\H$ can be collapsed to yield another blow-up of $\SS$, up to combinatorial isomorphism (we will provide a complete proof of this fact, see Corollary~\ref{renamepartition}). Disjoint subfamilies can be collapsed in any order with the same result.

The relation "having representatives which differ by a marking-preserving collapse" is a partial order on marked blow-ups of $\SS$, i.e.~elements of $K_\Gamma^{(0)}$, which is invariant under the action of $\unt(A_\Gamma)$. The \emph{spine of untwisted outer space} for $A_\Gamma$, denoted $K_\Gamma$, is the simplicial complex associated to this poset: simplices of dimension $k$ correspond to totally ordered sets of $k$ vertices. It is a finite-dimensional, locally finite, contractible complex, with a proper and cocompact action of $\unt(A_\Gamma)$ induced by the action on vertices.
\end{defi}

The following lemmas provide useful information on the structure of blow-ups of $\SS$.
\begin{lm}[\cite{twistedos}, Corollary~3.6]
\label{uniquecube}
In a blow-up $\SS^\mathbf{\Pi}$, maximal cubes are in one-to-one correspondance with maximal sets of pairwise adjacent labels in $\mathbf{\Pi}\sqcup V$. In particular, two labels are adjacent if and only if their hyperplanes are transverse.
\end{lm}
\begin{lm}
    Every blow-up $\SS^\mathbf{\Pi}$ is a special cube complex.
\end{lm}
Two-sidedness of hyperplanes comes from the fact that labels are oriented by construction of $\SS^\mathbf{\Pi}$. Also by construction of $\SS^\mathbf{\Pi}$, there is at most one edge with a given oriented label incident at each vertex, ruling out self-intersection and direct self-osculation of hyperplanes. Absence of inter-osculation comes from Lemma~\ref{uniquecube} as well as \cite{untwistedos}, Lemma~3.11: if two edges dual to transverse hyperplanes are incident to the same vertex, their labels are adjacent. Hence the edges are part of a $4$-cycle on which a squared is glued by construction of $\SS^\mathbf{\Pi}$.

\section{Properties of hyperplane collapses}

\begin{defi}
\label{defcollapse}
    Let $Z$ be a $\cat$ cube complex. Let $\W$ be the set of hyperplanes of $Z$ and $\H\subseteq \W$. The \emph{collapse of $\H$} is the restriction quotient $c\colon Z \to X(Z, \W\setminus \H)$.
\end{defi}
\begin{lm}
\label{collapse1}
    Let $Z$ be a $\cat$ cube complex with hyperplane set $\W$. Let $\H\subseteq \W$ and let $c \colon Z\to Y$ be the collapse of $\H$. The following hold:
    \begin{enumerate}
        \item $c$ is surjective, and cube preimages are convex.
        \item $c$ preserves medians.
        \item $c$ maps each $n$-cube of $Z$ onto a $k$-cube of $Y$, with $k\leq n$. Up to isometries, the restriction of $c$ to such an $n$-cube corresponds to the restriction $[0,1]^n\to [0,1]^k$ of the orthogonal projection $\RR^n\to \RR^k$ with kernel spanned by the basis vectors dual to hyperplanes of $\H$. 
        \item $c$ is the quotient map corresponding to the equivalence relation induced by the orthogonal projections on cubes defined above.
        \item Assume a group $G$ acts on $Z$ by combinatorial isometries and $\H$ is $G$-invariant. Then there is a natural action of $G$ on $Y$ making $c$ equivariant. If $G$ acts cocompactly on $Z$, then $G$ acts cocompactly on $Y$.
    \end{enumerate}
\end{lm}

\begin{proof}
    $(1)$  Recall that a $k$-cube $C$ in $Y$ corresponds to a choice of $k$ pairwise transverse hyperplanes $H_1,\dots, H_k\in \W\setminus \H$ together with a choice of halfspaces for all hyperplanes in $\W\setminus (\H\cup \{H_1,\dots,H_k\})$ such that every vertex of $Z$ lies in all but finitely many of the chosen halfspaces, every two chosen halfspaces intersect, and every chosen halfspace intersects with both halfspaces of every $H_i$. Note that the preimage $c^{-1}(C)$ is the intersection of all such chosen halfspaces, hence is convex. It remains to prove that it is non-empty.
    
    Let $C'\in Z$ be a $k$-cube dual to $H_1,\dots, H_k$. The existence of $C'$ is guaranteed by the fact that $H_1,\dots,H_k$ are pairwise transverse, and for each $H\in \W\setminus \{H_1,\dots, H_k\}$, $C'$ is entirely contained in a halfspace of $H$. Choose further $C'$ such that the (finite) number of chosen halfspaces for hyperplanes in $\W\setminus (\H\cup\{H_1,\dots,H_k\})$ that do not contain $C'$ is minimal. Assume the existence of $H\in \W\setminus (\H\cup\{H_1,\dots,H_k\})$ whose chosen halfspace does not contain $C'$. Note that in that case, both halfspaces of $H$ intersect both halfspaces of each of the $H_i$, meaning that $H,H_1,\dots, H_k$ are pairwise transverse. Taking such $H$ as close as possible from $C'$, $\s(C'\mid H)\setminus \H$ only contains hyperplanes $H'$ whose chosen halfspace contains $C'$. Such a hyperplane $H'$ cannot exist as its chosen halfspace would not intersect the chosen halfspace for $H$, contradicting the compatibility of chosen halfspaces. Therefore, $\s(C'\mid H) = \emptyset$, i.e.~there exists an edge $e$ dual to $H$ and incident to a vertex of $C'$. Since $H$ is transverse to all the $H_i$, $e$ spans a product with $C'$ by Lemma~\ref{square} and the link condition. The cube $C''$ parallel to $C'$ at the other end of $e$ is contained in one more chosen halfspace, contradicting the minimality of $C'$. Thus $C'$ lies in the chosen halfspace of every $H\in \W\setminus(\H\cup\{H_1,\dots,H_k\})$, i.e.~$c(C') = C$, proving surjectivity.

    $(2)$ By Lemma~\ref{edgepath}, $c$ maps geodesic edge paths to geodesic edge paths. Thus, for every vertices $z_1,z_2,z_3$ of $Z$, $c(\mu(z_1,z_2,z_3))$ lies on geodesics joining $c(z_i)$ with $c(z_j)$ for $i\neq j$. The result follows from uniqueness of the median in $Y$.

    $(3)$ Let $C$ be an $n$-cube of $Z$. $C_1\subseteq C$ be a maximal subcube that is not dual to any hyperplane of $\H$, with dimension $k$, and $C_2$ be a maximal subcube dual to only hyperplanes of $\H$. Then $C_2$ has dimension $n-k$, and $C_1$ and $C_2$ span $C$ as a product. Clearly, distinct vertices of $C_1$ are mapped to distinct vertices of $Y$ by $c$, and each copy of $C_2$ at a vertex of $C_1$ is mapped to a single vertex. Recall that $c$ is affine on $C$. This means that $c(C_1) = c(C)$, and $c$ is injective on $C_1$. Moreover, $c$ coincides with the orthogonal projection described in the statement on all vertices of $C$. Since $c$ is affine, $c$ equals that orthogonal projection.

    $(4)$ We prove that for $x,y\in Z$, the equality $c(x)=c(y)$ holds if and only if there exist a sequence of points $x=x_0,\dots,x_n=y$ such that for all $i$, there exists a cube $C$ containing $x_i$ and $x_{i+1}$ and the orthogonal projection from $C$ sends $x_i$ and $x_{i+1}$ to the same point. If there exists such a sequence $x_0,\dots, x_n$, it is straightforward that $c(x_i) = c(x_{i+1})$ for all $i$. Conversely, if $c(x) = c(y)$, let $C$ be the lowest-dimensional cube of $Y$ containing $c(x) = c(y)$. By the proof of $(1)$, $c^{-1}(C)$ decomposes as a product $F\times C$ where $F$ is a connected subcomplex with edges dual to hyperplanes of $\H$ only. The restriction of $c$ to $F\times C$ is the second coordinate projection. The points $x$ and $y$ write $(f_x,c(x))$, $(f_y, c(y))$ in these coordinates. By connectedness of $F$, there exists a sequence $f_x = f_0, \dots f_n = f_y$ of points of $F$ such that $f_i$ and $f_{i+1}$ are in a common cube for all $i$. Setting $x_i= (f_i, c(x))$ satisfies the assumptions.

    $(5)$ The action of $G$ on $Z$ induces an action on $\W\setminus \H$, and from there an action on the set of compatible choices of halfspaces for all hyperplanes in $\W\setminus \H$, i.e.~an action on $Y^{(0)}$. It is clear that this action extends to $Y$ and $c$ is equivariant. Since $c$ is surjective, it induces a surjection $Z/G\to Y/G$: if $Z/G$ is compact, $Y/G$ is compact as well.
\end{proof}

Recall that an isometry of a $\cat$ cube complex is a \emph{cube inversion} if it stabilizes some (finite-dimensional) cube without fixing it pointwise. It is a \emph{hyperplane inversion} if it stabilizes a hyperplane while exchanging its halfspaces.

\begin{defi}
\label{collapses}
    Let $G$ be a group acting by combinatorial isometries without cube inversions on a $\cat$ cube complex $Z$. Let $\W$ denote the set of hyperplanes of $Z$ with the induced action of $G$, and let $\H\subseteq\W$ be $G$-invariant. 
    We say that $\H$ is:
    \begin{itemize}
        \item \emph{weakly collapsible} if for every vertex $z$ of $Z$, and every $g\in G\setminus Stab(z)$, the separator $\s(z\mid gz)$ is not contained in $\H$;
        \item \emph{strongly collapsible} if for every $H\in \H$, and every $g\in G\setminus Stab(H)$, the separator $\s(H\mid gH)$ is not contained in $\H$.
    \end{itemize}
    A \emph{strong (resp. weak) collapse map} is the collapse of a strongly (resp. weakly) collapsible $G$-invariant family of hyperplanes.
\end{defi}

We will often implicitly consider collapse maps modulo $G$-equivariant combinatorial isomorphisms of their domain and range. Given a collapse (resp. weak collapse, strong collapse) $c\colon Z\to Y$ of a $G$-invariant family $\H$, we will often refer to induced map of quotients $Z/G\to Y/G$ as the collapse (resp. weak collapse, strong collapse) of the quotient family $\H/G$.

\begin{rem}
    The weak collapse condition is equivalent to asking that $Stab(c(z))\subseteq Stab(z)$ for every vertex $z$ of $Z$. The reverse inclusion is always true by equivariance of $c$. In particular, when $G$ acts freely on $Z$, $c$ is a weak collapse if and only if $G$ acts freely on $Y$, if and only if $c$ induces a homotopy equivalence $Z/G\to Y/G$ (see Lemma~\ref{collapse}).

    The strong collapse condition is equivalent to asking that for every hyperplane $H\in \H$ with carrier $K$, $c(K)$ is disjoint from its distinct $G$-translates in $Y$. It is also equivalent to asking that the immersion $c(K)/Stab(c(K)) \to Y/G$ is an embedding.
    
    When $G$ acts freely, it is clear from Lemma~\ref{collapse1} that the map of quotients $Z/G\to Y/G$ induced by $c$ is the same as the collapse defined in \cite{untwistedos} (Definitions~4.1 and 4.3) when the family $\H/G$ has compatible carriers (i.e.~when $\H$ is weakly collapsible and carriers of hyperplanes of $\H/G$ decompose as products in $Z/G$, one factor being the dual edge).

    Having compatible carriers in the quotient $Z/G$ implies being weakly collapsible, but the converse does not hold (the weak collapse condition does not require that carriers embed). When $G$ does not invert hyperplanes of $\H$ (i.e.~the corresponding quotient hyperplanes in $Z/G$ are two-sided), being strongly collapsible implies having compatible carriers in the quotient (see Lemma~\ref{strongimpliesweak}), but the converse does not hold either: consider for example a family of two parallel hyperplanes in the first subdivision of a single square with two opposite vertices identified.
    
    Note that our main interest will be free, cospecial actions, which have no cube inversions or hyperplane inversions.
\end{rem}

In Lemma~\ref{strongdownstairs}, we will give a characterization of strong collapses in the quotient. First, we need to establish some properties of collapse maps.

\begin{lm}
\label{strongimpliesweak}
    Let $G$ act on a $\cat$ cube complex $Z$ without cube inversions. Let $\H$ be a strongly collapsible family of hyperplanes in $Z$. Then $\H$ is weakly collapsible.
\end{lm}
\begin{proof}
    Let $z\in Z^{(0)}$, $g\in G$ and assume that $\s(z\mid gz)\subseteq \H$, yet $gz\neq z$. Let $H$ be dual to the first edge of a geodesic edge path $p$ joining $z$ to $gz$. Then $\s(H\mid gH)\subseteq \s(z\mid gz)\subseteq \H$. Therefore $g\in Stab(H)$ by the strong collapse assumption. Since $p$ is geodesic, $z$ and $gz$ are not in the same halfspace of $H$, thus $g$ exchanges the halfspaces of $H$. By convexity, $p$ lies in the carrier of $H$ entirely. If $p$ is a single edge, this edge is a cube inverted by $g$, a contradiction. Otherwise, there exists another geodesic path $p'$ in the carrier of $H$ with the same endpoints as $p$ whose first edge is dual to a hyperplane transverse to $H$. Proceeding inductively in the $g$-invariant carrier of $H$, we get that all hyperplanes in $\s(z\mid gz)$ are pairwise transverse, and stabilized by $g$ with exchanged halfspaces. Thus $g^2z = z$ and the union of geodesic edge paths joining $z$ to $gz$ is then the $1$-skeleton of a cube of $Z$ inverted by $g$, a contradiction.
\end{proof}

\begin{lm}
\label{collapse}
    Let $G$ act on a $\cat$ cube complex $Z$ without cube inversions. Let $\H$ be a weakly collapsible family of hyperplanes in $Z$ with collapse $c\colon Z\to Y$. The following hold:
    \begin{enumerate}
        \item The action of $G$ on $Y$ is without cube inversions.
        \item If the action on $Z$ is free, so is the action on $Y$.
        \item If the action on $Z$ is  free, $c$ induces a homotopy equivalence $Z/G\to Y/G$ whose vertex preimages are locally convex subcomplexes in $Z/G$ and $\cat$.
    \end{enumerate}
\end{lm}

\begin{proof}
    $(1)$ Assume $g\in G$ stabilizes a cube $C$ of $Y$. Let $\mathcal{C}$ be the finite collection of pairwise transverse hyperplanes corresponding to $C$, so that $g\mathcal{C} = \mathcal{C}$. Some power $g^n$ fixes $C$ pointwise. Since $C$ is convex, it corresponds to an intersection of halfspaces of hyperplanes of $\W\setminus \H$. Then the preimage $D = c^{-1}(C)$ is convex in $Z$, and non-empty by $(1)$. Moreover, since $c$ is equivariant, $gD = D$. The carriers in $Z$ of all hyperplanes of $\mathcal{C}$ intersect $D$, thus by the Helly property they intersect globally in $D$. Let $A$ denote the intersection of all such carriers and $D$. By Corollary~\ref{product}, $A$ decomposes as a product $A_1\times A_2$ where $A_1$ is a cube with set of dual hyperplanes equal to $\mathcal{C}$ and $A_2$ is a connected complex whose dual hyperplanes are all in $\H$. Since $g\mathcal{C} = \mathcal{C}$, $gA = A$. For every vertex $z$ of $A$, $g^n\in Stab(c(z)) = Stab(z)$ by the weak collapse assumption. Thus $g$ acts on $A$ with finite orbits. Therefore, $g$ stabilizes a cube of $A$ (see \cite{bridsonhaefliger}, II, Corollary~2.8 for a $\cat$ proof, or \cite{genevoisbook}, Theorem~4.1.1 for a combinatorial proof). Since $g$ is not a cube inversion of $Z$, $g$ fixes some vertex $z\in A$. Let $C'$ be the copy of $A_1$ at $z$, i.e.~the cube spanned at $z$ by edges dual to all hyperplanes of $\mathcal{C}$. Since $g\mathcal{C} = \mathcal{C}$, $gC' = C'$, and $g$ is not a cube inversion, hence $g$ fixes $C'$ pointwise. By construction, $c(C') = c(A_1) = C$, and $c$ is equivariant, hence $g$ fixes $C$ pointwise.

    $(2)$ This is clear since both actions are without cube inversions and have the same vertex stabilizers by the weak collapse assumption.

    $(3)$ When both actions are free, $c$ induces a homotopy equivalence $Z/G\to Y/G$ and vertex preimages of $Z\to Y$ embed in $Z/G$ via the quotient map, proving the result.
\end{proof}

We can now prove the characterization.
\begin{lm}
\label{strongdownstairs}
    Let $G$ act freely on a $\cat$ cube complex $Z$ by combinatorial isometries, and let $c$ be any (equivariant) collapse of $Z$. Then $c$ is a strong collapse if and only if $c$ induces a homotopy equivalence in the quotient, and for every vertex preimage $C$ in $Z/G$, every two edges of $C$ that are parallel in $Z/G$ are parallel in $C$.
\end{lm}

\begin{proof}
    Let $\H$ be the family of hyperplanes collapsed by $c$. Assume first that $\H$ is strongly collapsible. Then $\H$ is weakly collapsible by Lemma~\ref{strongimpliesweak} and $c$ induces a homotopy equivalence by Lemma~\ref{collapse}. Choose $C\subseteq Z/G$ a vertex preimage, and $e,f$ two edges of $C$ with the same dual hyperplane $\overline{H}$ in $Z/G$. By Lemma~\ref{collapse1} , $C$ is $\cat$. Choose $p$ a shortest edge path in $C$ from any endpoints of $e$ to any endpoint of $f$. In particular, $p$ does not contain $e$ or $f$.
    
    Assume first that $p$ has no edge dual to $\overline{H}$. Lift $p$ to a path $\uc{p}$ in $Z$ between lifts $\uc{e}, \uc{f}$ of $e$ and $f$. By assumption, $\uc{p}$ is only dual to hyperplanes of $\H$. Moreover, the dual hyperplanes to $\uc{e}$ and $\uc{f}$ are in $\H$ as well and both project to $\overline{H}$ in the quotient. Therefore, these hyperplanes are of the form $H,gH$ for some $g\in G$. Every hyperplane in $\s(H\mid gH)$ is dual to an edge of $\uc{p}$, hence lies in $\H$. By the strong collapse assumption, $gH=H$. Since $p$ is geodesic, so is $\uc{p}$, hence $\uc{p}$ remains in the carrier of $H$ by convexity. Moreover, since $\uc{p}$ has no edge dual to $H$, $\uc{p}$ remains in one halfspace of $H$. By Corollary~\ref{product}, $\uc{p}$ spans a product with edges dual to $H$, the first and last of those edges being $\uc{e}$ and $\uc{f}$ respectively. In the quotient, $p$ spans a product with edges dual to $H$, the first and last of those edges being $e$ and $f$ respectively, proving that $e$ and $f$ are parallel in $C$.

    Now assume $p$ has edges $e_1,\dots e_n$ dual to $\overline{H}$, appearing in $p$ in that order. Set $e_0 = e$ and $e_{n+1} = f$. Then for each $i$, the edges $e_i$ and $e_{i+1}$ are parallel in $C$ by the previous argument. This proves that $e$ and $f$ are parallel in $C$.

    For the converse, assume that $c$ induces a homotopy equivalence and that every pair of edges in a vertex preimage that are parallel in $Z/G$ are parallel in that vertex preimage. Let $H\in \H$, $g\in G$, with $\s(H\mid gH)\subseteq \H$. Let $e,f$ be edges dual to $H,gH$ respectively such that $\s(e\mid f)\subseteq \H$. The collapse $c$ collapses $e$ and $f$ to the same vertex. Therefore, their images $\overline{e},\overline{f}$ in $Z/G$ belong to the same vertex preimage $C$. Since $H$ and $gH$ have the same quotient hyperplane $\overline{H}$ in $Z/G$, $\overline{e}$ and $\overline{f}$ are parallel in $Z/G$. By assumption, $\overline{e}$ and $\overline{f}$ are parallel in $C$, therefore there exists a sequence of squares in $C$ witnessing that parallelism, having opposite edges $\overline{e}_i$ and $\overline{e}_{i+1}$, with $\overline{e}_0 = \overline{e}$ and $\overline{e}_n = \overline{f}$. This sequence of squares, seen as a map from $[0,1]\times [0,n]$ to $Z/G$, lifts to a sequence of squares in $Z$ starting with the edge $e$. The terminal edge is of the form $g'f$, with $g'\in G$ and parallel to $e$, thus dual to $H$. Moreover, all the squares lift squares of $C$, meaning that $\s(e\mid g'f)\subseteq \H$. Therefore, $\s(g'f\mid f)\subseteq \H$ and for any endpoint $z$ of $f$, $c(z) = c(g'z)$. Since $c$ induces a homotopy equivalence and $G$ acts freely, this implies that $g'z=z$ and $g'$ is the identity. Thus $g'f = f$ is dual to both $H$ and $gH$: $g$ stabilizes $H$.
\end{proof}

We quote the following lemma, a straightforward consequence of the definitions.

\begin{lm}[Partial strong collapses are strong]
\label{subfamilystrong}
    Let $\H$ be a family of hyperplanes of a $\cat$ cube complex $Z$ with a partition $\H=\H_1\sqcup \H_2$ into $G$-invariant hyperplane families. Then $\H_2$ identifies to a $G$-invariant family of hyperplanes in the range $Y$ of the collapse of $\H_1$. The collapse of $\H$ in $Z$ factors as the composition of the collapse of $\H_1$ in $Z$ followed by the collapse of $\H_2$ in $Y$. If $\H$ is strongly (resp. weakly) collapsible in $Z$, then $\H_1$ is strongly (resp. weakly) collapsible in $Z$ and the image of $\H_2$ is strongly (resp. weakly) collapsible in $Y$.
\end{lm}

\begin{lm}[Strong collapses are quasi-isometries]
\label{quasiisometry}
    Let $G$ act on a $\cat$ cube complex $Z$. Let $c\colon Z\to Y$ be the strong collapse of a family $\H$ containing a finite number $n$ of hyperplane orbits. For $z_1,z_2$ vertices of $Z$, the following holds:
    \[ \frac{d_Z(z_1,z_2)-n}{n+1}\leq d_Y(c(z_1),c(z_2))\leq d_Z(z_1,z_2)\]
    In particular, vertex preimages of $c$ have diameter at most $n$.
\end{lm}

Note that the lemma makes no cocompactness assumption. This fails for a weak collapse: $\ZZ = \gen{(1,1)}$ acts (non-cocompactly) on the standard cubulation of $\RR^2$ by translations with two hyperplane orbits. Yet, the weak collapse all horizontal hyperplanes is a linear map of rank $1$, hence not a quasi-isometry.

\begin{proof}
    Let $z_1,z_2$ be vertices of $Z$. Let $p$ be a geodesic edge path joining $z_1$ and $z_2$. Exactly $d_Y(c(z_1),c(z_2))$ edges of $p$ are dual to hyperplanes outside $\H$, proving the second inequality. Moreover, $p$ cannot contain $n+1$ consecutive edges dual to hyperplanes of $\H$, otherwise two of those edges would be dual to distinct hyperplanes of $\H$ in the same orbit and yet separated only by hyperplanes of $\H$, a contradiction with the strong collapse assumption. Therefore, $p$ has at most $n(d_Y(c(z_1),c(z_2))+1)+d_Y(c(z_1),c(z_2))$ edges, proving the first inequality.
\end{proof}

The following result will be instrumental to the proof of Proposition~\ref{spatialblowup}.

\begin{lm}
\label{tube}
Let $G$ act freely and cocompactly on a $\cat$ cube complex $Z$ by combinatorial isometries without hyperplane inversions. Let $\uc{c}\colon Z\to Y$ be a strong collapse and $c\colon Z/G\to Y/G$ the induced homotopy equivalence between quotients. Assume that $Y/G$ is combinatorially isomorphic to the Salvetti $\SS$ (i.e.~$G\simeq A_\Gamma$ acts on $Y$ cospecially and transitively on vertices).

Let $C\subseteq Z/G$ be the unique vertex preimage. Let $e$ be an edge of $(Z/G)\setminus C$ with dual hyperplane $H_e$, and consider $p$ a geodesic edge path in $C$ joining the endpoints of $e$. The following hold:
\begin{itemize}
    \item Every edge $f$ spanning a square with $e$ spans a product with the cycle $ep^{-1}$, in the sense that the square complex $f\times ep^{-1}$ immerses in $X$, with $f\times e$ sent to the square and $\{x\}\times ep^{-1}$ sent to $ep^{-1}$ for $x$ the appropriate endpoint of $f$. 
    \item Every hyperplane transverse to $H_e$ is transverse to all the hyperplanes dual to edges of $p$.
\end{itemize}
\end{lm}

\begin{proof}
    Lift the cycle $ep^{-1}$ to a path $\uc{e}\uc{p}^{-1}$ in $Z$ with endpoints $x,gx$ for some $g\in G$. Let $\K\subseteq \H$ be the set of hyperplanes dual to edges of $\uc{p}$, and let $H$ be the hyperplane dual to $\uc{e}$, lifting $H_e$. Let $f$ be any edge spanning a square with $e$ and $\uc{f}$ the lifted edge at $x$ spanning a square with $\uc{e}$, with dual hyperplane $H'$. Write $x'$ the other endpoint of $\uc{f}$. Two cases arise:
    \begin{itemize}
        \item If $H'\in \H$, $\s(H'\mid gH')\subseteq \s(x\mid gx) = \K\cup \{H\}$, and $H\notin \s(H'\mid gH')$ since $H$ and $H'$ are transverse. Therefore, $\s(H'\mid gH')\subseteq \K\subseteq\H$ and $g\in Stab(H')$ by the strong collapse assumption.
        \item If $H'\notin \H$, let $y,y'$ be the other vertices of the square spanned by $\uc{e}$ and $\uc{f}$, so that $\s(x\mid y) = \s(x'\mid y') = \{H\}$. Then $\s(y\mid gx) = \K\subseteq \H$, meaning that $\uc{c}(y) = \uc{c}(gx)$. The two hyperplanes $\uc{c}(H')$ and $\uc{c}(gH')$ are both incident to the same vertex $\uc{c}(y) = \uc{c}(x)$ with the same orientation. Since $G$ acts cospecially on $Y$, the two hyperplanes must be equal, hence $g\in Stab(H')$. 
    \end{itemize}
    In both cases, since $g$ stabilizes $H'$, the carrier of $H' = gH'$ contains both endpoints of $\uc{p}$. Since $g$ does not invert $H'$ those endpoints are in the same halfspace of $H'$. By convexity, the entire path $\uc{e}\uc{p}^{-1}$ is contained in the carrier of $H'$, in the same halfspace of $H'$. Thus, the path $\uc{e}\uc{p}^{-1}$ spans a product with the edge $\uc{f}$, by Corollary~\ref{product}. Note that the edge of that product dual to $H'$ at $gx$ must be $g\uc{f}$ by the same lemma. Therefore, this product descends to an immersed product of $f$ and the cycle $ep^{-1}$ in $Z/G$, proving the first statement.

    Let $H_{f'}$ be a hyperplane transverse to $H_e$, dual to some edge $f'$ of the carrier of $H_e$. There exists a path $f_1\dots f_n$ in the carrier of $H_e$ such that each edge $f_i$ spans a square with an edge dual to $H_e$, $f_1$ spans a square with $e$ and $f_n = f'$. The second statement is proved by induction on $n$: if $n=1$, it is a direct consequence of the first statement. Otherwise, by the first statement, $f_1$ spans a product with the cycle $ep^{-1}$. Consider then the cycle $e'p'^{-1}$ opposite $ep^{-1}$ in that product. The edge $e'$ is dual to $H_e$, not in $C$, and $p'$ is dual to the same hyperplanes as $p$, in particular $p'$ is a geodesic in $C$. The induction then proceeds using $e'$ and $p'$ to apply the first statement.
\end{proof}

\section{Collapses between blow-ups}

In this section we prove that the collapse of a tree-like families of hyperplanes in a blow-up $\SS^\mathbf{\Pi}$ is an example of strong collapse (see Definition~\ref{defspine}). By Lemma~\ref{subfamilystrong}, the same holds for any sub-family of a tree-like family. Conversely, any collapse whose domain and range are both blow-ups of the Salvetti $\SS$ is the collapse of a sub-family of a tree-like family. In particular, collapses between blow-ups of $\SS$ are strong.

\begin{lm}[Tree-like collapses are strong collapses]
\label{blowupstrong}
    Let $\mathbf{\Pi}$ be a collection of pairwise compatible Whitehead partitions and $\overline{\H}$ a tree-like family of hyperplanes in the blow-up $\SS^\mathbf{\Pi}$. Let $\H$ be the set of all lifts of hyperplanes of $\overline{\H}$ to the universal cover $\uc{\SS^\mathbf{\Pi}}$ endowed with the free cocompact action of $G=\pi_1\SS^\mathbf{\Pi}\simeq A_\Gamma$. Then $\H$ is strongly collapsible, with strong collapse map lifting the tree-like collapse $\SS^\mathbf{\Pi}\to\SS$ of $\overline{\H}$.
\end{lm}

\begin{proof}
    By \cite{untwistedos}, Theorem~4.11, $\overline{\H}$ has compatible carriers. In particular, $\H$ is weakly collapsible, inducing the tree-like collapse $\SS^\mathbf{\Pi}\to \SS$ of $\overline{H}$ in the quotient.

    Let $H\in \H, g\in G$ and assume $\s(H\mid gH)\subseteq \H$. Let $x,y$ be vertices of the carrier of $H$, in the same halfspace of $H$, such that $\s(x\mid gy)\subseteq \H$. Let $p$ be a geodesic edge path joining $x$ and $y$, remaining in the carrier of $H$ by convexity, and let $q$ be its projection in the quotient $\SS^\mathbf{\Pi}$. The path $q$ remains in the carrier of the projected hyperplane $\overline{H}$ and never crosses $\overline{H}$. Let $l$ denote the label of $\overline{H}$ in $\SS^\mathbf{\Pi}$. By construction of $\SS^\mathbf{\Pi}$, for every edge $e$ of $q$ labelled by a standard generator $v\in V$, there exists an edge path $q_e$ with the same endpoints as $e$ such that every label of an edge of $q_e$ is a partition $\mathbf{P}$ with $v\in single(\mathbf{P})$. In particular, since $l\in lk(v)$, $l\in lk(\mathbf{P})$, meaning that the path $q_e$ remains in the carrier of $\overline{H}$. Let $q'$ be the path obtained from $q$ by replacing each such edge $e$ with the corresponding path $q_e$. This path $q'$ has the same endpoints as $q$, still lies in the carrier of $\overline{H}$, and all of its edges are labelled by partitions.

    We make the following claim: For every edge $f$ labelled by some partition $\mathbf{P}$ with $l\in lk(\mathbf{P})$, there exists a path $q'_f$ with the same endpoints as $f$ such that the following holds: Every edge of $q'_f$ is either dual to a hyperplane of $\overline{\H}$ transverse to $\overline{H}$, or is labelled by a partition $\mathbf{Q}$ with $l\in lk(\mathbf{P})\subsetneq lk(\mathbf{Q})$. In particular, $q'_f$ remains in the carrier of $\overline{H}$. Indeed, let $\mathbf{\Pi}'\subseteq \mathbf{\Pi}$ denote the set of partitions that have the same link as $\mathbf{P}$, and let $\overline{f}$ be the projection of $f$ via the collapse $\SS^\mathbf{\Pi}\to \SS^{\mathbf{\Pi}'}$. By definition of a tree-like family of hyperplanes, there exists a path $\gamma$ joining the endpoints of $\overline{f}$ in $\SS^{\mathbf{\Pi}'}$ such that every edge of $\gamma$ is dual to a hyperplane projected from $\overline{\H}$, whose label has the same link as $\mathbf{P}$. Define $q'_f$ as the shortest path in $\SS^\mathbf{\Pi}$ with the same endpoints as $f$ that projects to $\gamma$ in $\SS^\mathbf{\Pi'}$. Every edge of $q'_f$ is either dual to a hyperplane of $\overline{\H}$ transverse to $\overline{H}$, or labelled by a partition $\mathbf{Q}\in \mathbf{\Pi}\setminus\mathbf{\Pi}'$. Assume $\mathbf{Q}\in \mathbf{\Pi}\setminus\mathbf{\Pi}'$ appears as the label of an edge of $q'_f$. The cycle $q'_f f$ is a shortest cycle lifting $\gamma \overline f$ to $\SS^\mathbf{\Pi}$ and has an edge labelled $\mathbf{Q}$. By construction of $\SS^\mathbf{\Pi}$, some edge of $q'_f f$ is labelled by a vertex $v\in single(\mathbf{Q})$. In particular, $lk(v)\subseteq lk(\mathbf{Q})$. Since $v$ is not the label of any hyperplane of $\overline{\H}$ by definition, $v$ is the label of an edge of $\gamma \overline{f}$. In particular, $lk(v) = lk(\mathbf{P})$. This proves that $lk(\mathbf{P})\subseteq lk(\mathbf{Q})$ and equality cannot hold since $\mathbf{Q}\notin \mathbf{\Pi}'$. This proves the claim.
    
    Let $f$ be an edge of $q'$ which is not dual to a hyperplane of $\overline{\H}$. Choose such an $f$ whose label has a link of minimum size. Now use the claim to replace $f$ by the corresponding path $q'_f$ (remaining in the carrier of $\overline{H}$), and iterate this process inductively. The process terminates because at each step, either the number of edges labelled by partitions with link of minimal size decreases, or the minimal size increases (that size being bounded). The path $q''$ on which the process terminates has the same endpoints as $q'$ and $q$, is still contained in the carrier of $\overline{H}$ and crosses only hyperplanes of $\overline{\H}$. The cycle $q^{-1}q''$ lifts to a path $p^{-1}p''$ in the universal cover $\uc{\SS^\mathbf{\Pi}}$ starting from $y$, with other endpoint $g'y$ for some $g'\in G$.

    Finally, $\s(g'y\mid gy)\subseteq \s(g'y\mid x)\cup \s(x\mid gy)\subseteq \H$. By the weak collapse assumption, $gy = g'y$. This point, an endpoint of $p''$, belongs to the carrier of $H$. The oriented edges dual to $H$ containing $y$ and $gy$ project to parallel oriented edges incident to the same vertex of $\SS^\mathbf{\Pi}$ with the same orientation. By construction of $\SS^\mathbf{\Pi}$, they must be the same edge in the quotient, hence $g\in Stab(H)$. Since $g\in G$ and $H\in\H$ were arbitrary, $\H$ is strongly collapsible.
\end{proof}

\begin{lm}[Collapses between blow-ups are tree-like]
\label{weakblowupgood}
    Let $\overline{\H}$ be a family of hyperplanes in a blow-up $\SS^\mathbf{\Pi}$. Assume that the range of the collapse of $\overline{\H}$ is isomorphic to a blow-up $\SS^{\mathbf{\Pi}'}$ (of the same Salvetti complex $\SS$). Then $\overline{\H}$ is contained in a tree-like family of hyperplanes of $\SS^\mathbf{\Pi}$.
\end{lm}
The assumption that the range is isomorphic to a blow-up of $\SS$ is necessary. Indeed, in the single blow-up of $\gen{a,b,c\mid [a,b] = 1}$ corresponding to the partition $\{a,c\}\sqcup \{a^{-1},c^{-1}\}\sqcup \{b,b^{-1}\}$ with basepoint $a$, the family containing only the hyperplane dual to $c$ is weakly collapsible but not tree-like.
\begin{proof}
The collapse of $\overline{\H}$ lifts to the collapse of some $A_\Gamma$-invariant family $\H$ between universal covers. Since this collapse is equivariant and both actions are free, it is a weak collapse. The composition $\SS^\mathbf{\Pi}\to \SS^{\mathbf{\Pi}'}\to \SS$ of the collapse of $\overline{\H}$ and the canonical collapse of hyperplanes labelled $\mathbf{\Pi}'$ is a homotopy equivalence. Thus, up to adjoining $\overline{\H}$ with all hyperplanes of $\SS^\mathbf{\Pi}$ mapping to partition-labelled hyperplanes in $\SS^{\mathbf{\Pi}'}$, assume $\mathbf{\Pi}'$ is empty. The family $\overline{\H}$ still lifts to a weakly collapsible family, and the range of its collapse is isomorphic to the Salvetti complex $\SS$. The goal is to apply Theorem~\ref{treelikecompatible} to prove that $\overline{\H}$ is tree-like.

By the weak collapse assumption, every edge cycle in $\SS^\mathbf{\Pi}$ dual to hyperplanes of $\overline{\H}$ only is nullhomotopic. It only remains to check that hyperplanes of $\overline{\H}$ are carrier retracts in $\SS^\mathbf{\Pi}$, i.e.~have carriers that embed as a product with the dual edge. Assume $H\in\overline{\H}$ is not a carrier retract. By specialness of $\SS^\mathbf{\Pi}$, there exists $\gamma$ a closed edge path in the carrier of $H$ that crosses $H$ exactly once. By the weak collapse assumption, no edge dual to $H$ is a loop, hence the label $l$ of $H$ is either a partition of $\mathbf{\Pi}$ or split by some partition of $\mathbf{\Pi}$. Let $\mathbf{P}\in \mathbf{\Pi}$ be the partition either equal to $l$ or splitting $l$ and note that $lk(l)\subseteq lk(\mathbf{P})$. The image $H'$ of $H$ in $\SS^\mathbf{P}$ under the collapse of $\mathbf{\Pi}\setminus \mathbf{P}$ is still not dual to any loops. The image $\gamma'$ of $\gamma$ under this collapse is a closed path in $\SS^\mathbf{P}$. Yet, all the edges of $\gamma'$ except the unique edge dual to $H'$ have labels in $lk(l)\subseteq lk(\mathbf{P})$. By construction of $\SS^\mathbf{P}$, these edges are loops, and $\gamma'$ cannot be closed, a contradiction.
\end{proof}

\section{Spatial cube complexes}

The graphs involved in the construction of Culler-Vogtmann's outer space for free groups are assumed not to contain vertices of valence $2$, i.e.~are maximally unsubdivided. We describe an analogue using hyperplanes.
\begin{defi}
\label{redundantdef}
    Let $n\geq 1$. The \emph{subdivision} of the standard $n$-cube $[0,1]^n$ along the hyperplane given by the first coordinate is the product cube complex $[0,1]\times [0,1]^{n-1}$ where the first factor has vertex set $\{0,1/2,1\}$ and the second factor is the standard $n-1$-cube. It comes with a cellular homeomorphism from the standard $n$-cube given by the identity map of underlying sets (but its inverse is not cellular).

    Let $X$ be a locally $\cat$ cube complex and $H$ a two-sided, non-self-intersecting hyperplane of $X$. The \emph{subdivision of $X$ along $H$} is the locally $\cat$ cube complex $X'$ obtained by applying the subdivision procedure described above to each cube of $X$ containing one edge dual to $H$, along the hyperplane $H$. There is a cellular homeomorphism $X\to X'$, which is piecewise affine. The complex $X'$ is locally $\cat$. Any edge dual of $X$ dual to $H$ is mapped to a path of length $2$ in $X'$. The two hyperplanes dual to this path are the \emph{images of $H$ in $X'$}. They are distinct because $H$ is two-sided.

    Finally, let $X$ be any locally $\cat$ cube complex. A pair $\{H_1,H_2\}$ of distinct hyperplanes of $X$ is \emph{redundant} if there exists a locally $\cat$ cube complex $Y$ with a two-sided, non-self-intersecting hyperplane $H$, and an isomorphism between $X$ and the subdivision $Y'$ of $Y$ along $H$, mapping $H_1,H_2$ to the images of $H$ in $Y'$. Note that redundant hyperplanes are non-transverse, and have the same set of transverse hyperplanes.
\end{defi}

\begin{rem}
\label{collapseredundant}
    If $\{H_1,H_2\}$ is a redundant pair of non-self-osculating hyperplanes in a locally $\cat$ cube complex $X$, it is clear from the definition and Lemma~\ref{strongdownstairs} that the family of lifts of $H_1$ (resp. of $H_2$) to the universal cover is strongly collapsible. Letting $X\to X_1$, $X\to X_2$ be the corresponding collapses in the quotient, there is an isomorphism $X_1\to X_2$ such that the composition $X\to X_1\to X_2$ is homotopic to the collapse $X\to X_2$.
    
    Moreover, it is clear from the definition and the description in Lemma~\ref{collapse1} that in the range of the collapse of any family of hyperplanes in $X$ disjoint from $\{H_1,H_2\}$, the pair of images of $H_1$ and $H_2$ is still redundant.
\end{rem}
\begin{lm}
\label{blowupnoredundant}
    Let $\mathbf{\Pi}$ be a family of pairwise compatible Whitehead partitions. The corresponding blow-up $\SS^\mathbf{\Pi}$ of $\SS$ has no redundant pair of hyperplanes.
\end{lm}
\begin{proof}
    If $\mathbf{\Pi}$ is empty, the result is obvious since $\SS$ has only one vertex. If $\mathbf{\Pi} =\{\mathbf{P}\}$ is a single Whitehead partition, the single blow-up $\SS^\mathbf{P}$ has no redundant pair either. Indeed, since $P$ and $P^*$ have at least two elements each, both vertices of $\SS^\mathbf{P}$ have at least three incident edges that pairwise do not span squares (one edge labelled $\mathbf{P}$, one edge labelled by a basepoint $b$ of $\mathbf{P}$ and one edge labelled by an element of $P\cup P^*\setminus \{b,b^{-1}\}$).
    
    If $\mathbf{\Pi} =\{\mathbf{P,Q}\}$, every pair distinct from the pair $\{H_\mathbf{P}, H_\mathbf{Q}\}$ of hyperplanes labelled $\mathbf{P}$ and $\mathbf{Q}$ is preserved in one of the collapses $\SS^{\mathbf{P}}$, $\SS^\mathbf{Q}$, therefore not redundant by Remark~\ref{collapseredundant}. If $H_\mathbf{P}, H_\mathbf{Q}$ are transverse, they do not have the same set of transverse hyperplanes, therefore are not redundant. If $H_\mathbf{P}, H_\mathbf{Q}$ are not transverse, $\mathbf{P}$ and $\mathbf{Q}$ are not adjacent. Hence, exactly one of the four quadrants, say $P\cap Q^*$ is empty. It is then easily checked that all three vertices of $\SS^{\{\mathbf{P},\mathbf{Q}\}}$ have at least three incident edges that pairwise do not span squares: for the vertex corresponding to $P\cap Q$ (resp. $P^*\cap Q^*$), take edges labelled by $\mathbf{P}$, a basepoint $b$ of $\mathbf{P}$, and an element of $P\setminus \{b\}$ (resp. $\mathbf{Q}$, a basepoint $b'$ of $\mathbf{Q}$ and an element of $Q^*\setminus \{b'^{-1}\}$). Finally, for the vertex corresponding to $P^*\cap Q$, take edges labelled by $\mathbf{P}$, $\mathbf{Q}$ and an element of $P^*\cap Q$. This proves that $\{H_\mathbf{P}, H_\mathbf{Q}\}$ is not redundant.

    Finally, any pair of hyperplanes $\{H_1,H_2\}$ in any blow-up $\SS^\mathbf{\Pi}$ of $\SS$ is preserved in the range of the collapse of all hyperplanes different from $H_1$ and $H_2$, and labelled by $\mathbf{\Pi}$. This collapse is a Salvetti complex, a single blow-up or a double blow-up, hence the images of $H_1$, $H_2$ do not form a redundant pair. By Remark~\ref{collapseredundant}, $\SS^\mathbf{\Pi}$ contains no redundant pair of hyperplanes.
\end{proof}

\begin{defi}
\label{defspatial}
    Consider a free cocompact action of $A_\Gamma$ on a $\cat$ cube complex $Z$ by combinatorial isometries. This action is:
    \begin{itemize}
        \item \emph{Salvetti-like} if the quotient $Z/A_\Gamma$ is combinatorially isomorphic to the Salvetti complex $\SS$, or equivalently if the action is transitive on vertices and cospecial;
        \item \emph{blow-up-like} if $Z/A_\Gamma$ is combinatorially isomorphic to some blow-up $\SS^\mathbf{\Pi}$, where $\mathbf{\Pi}$ is a family of pairwise compatible Whitehead partitions of $V^\pm$;
        \item  \emph{cospatial} if $Z/A_\Gamma$ has no redundant pair of hyperplanes, and for every hyperplane $H$ of $Z$, there exists a strong collapse map with Salvetti-like range that does not collapse $H$.
    \end{itemize}
    A \emph{spatial cube complex} is the quotient of a cospatial action.
\end{defi}

\begin{rem}
\label{blowupspatial}
     Let $\mathbf{\Pi}$ be a family of pairwise compatible Whitehead partitions, and consider the blow-up $\SS^\mathbf{\Pi}$ of $\SS$. For every partition $\mathbf{P}\in \mathbf{\Pi}$, there exists a basepoint $b$ for $\mathbf{P}$ making the family $(\mathbf{\Pi}\setminus \{\mathbf{P}\})\cup \{b\}$ tree-like. Lemma~\ref{blowupstrong} guarantees that tree-like families are strongly collapsible, and by \cite{untwistedos}, Theorem~4.11, the range of such a collapse is combinatorially isomorphic to the Salvetti complex $\SS$. Moreover, $\SS^\mathbf{\Pi}$ has no redundant pair of hyperplanes by Lemma~\ref{blowupnoredundant}. Therefore, every such blow-up is spatial. Proposition~\ref{spatialblowup} will prove the converse.
\end{rem}

\begin{lm}
\label{spatialspecial}
    Every spatial cube complex is special.
\end{lm}

Note that the converse is false: the action of $\ZZ^2 = \gen{(1,1), (-1,1)}$ on $\RR^2$ endowed with its standard cubical subdivision is cospecial, but it is not cospatial.

\begin{proof}
    Let $A_\Gamma$ act cospatially on a $\cat$ cube complex $Z$. For every hyperplane $H$ of $Z$, there exists a strong collapse $c\colon Z\to S$ with $S$ Salvetti-like that does not collapse $H$. Since $A_\Gamma$ acts on $H$ cospecially, if $g\in A_\Gamma$ stabilizes $H$, $g$ does not exchange the halfspaces of $H$, both in $S$ and in $Z$. Moreover, no translate of $H$ is transverse to or directly osculating $H$, both in $S$ and in $Z$ (see Definition~\ref{special}). In particular, $A_\Gamma$ acts without hyperplane inversion.

    Assume the quotient hyperplane $\overline{H}$ is inter-osculating another hyperplane $\overline{K}$ of $Z/A_\Gamma$, and let $C\subseteq Z/A_\Gamma$ be the unique vertex preimage of the collapse. There exists two vertices $x,y$ of $Z/A_\Gamma$ such that $\overline{H}$ and $\overline{K}$ osculate at $x$ and are transverse at $y$. Let $e$, $f$ be the edges dual to $\overline{H}$, $\overline{K}$ at $x$ that witness the osculation. Up to changing $y$, assume that there exists a path $p$ from $x$ to $y$ spanning a product with $e$ and a path $q$ joining $x$ to $y$ spanning a product with $f$. By Lemma~\ref{tube}, every edge of $p$ (resp. $q$) outside $C$ can be replaced by a path in $C$ still spanning a product with edges dual to $\overline{H}$ (resp. $\overline{K}$) with the same endpoints and edges dual to $\overline{H}$ (resp. $\overline{K}$) at those endpoints. In other words, $p$ and $q$ can be assumed contained in $C$. Since $C$ is $\cat$, there exists lifts $\uc{p}$ and $\uc{q}$ of $p$ and $q$ to $Z$ with the same endpoints $\uc{x}$ and $\uc{y}$. By construction, $\uc{p}$ lies in the carrier of a lift $H$ of $\overline{H}$ and $\uc{q}$ lies in the carrier of a lift $K$ of $\overline{K}$. Moreover, $H$ and $K$ are transverse at $\uc{y}$, $H$ is dual to the lift $\uc{e}$ of $e$ at $\uc{x}$ and $K$ is dual to the lift $\uc{f}$ of $f$ at $\uc{x}$. By Lemma~\ref{square}, $\uc{e}$ and $\uc{f}$ span a square at $\uc{x}$, hence $e$ and $f$ span a square at $x$, contradicting the osculation.
\end{proof}

\begin{pro}
\label{spatialblowup}
    Every strong collapse with cospatial domain and Salvetti-like range is isomorphic to a canonical collapse of the form $c_\mathbf{\Pi}\colon \SS^\mathbf{\Pi}\to \SS$. In particular, every cospatial action of $A_\Gamma$ is blow-up-like.
\end{pro}

\begin{proof}
    Let $\uc{c}\colon \uc{X} \to \uc{S}$ be a strong collapse with $\uc{X}$ cospatial and $\uc{S}$ Salvetti-like. Let $X=\uc{X}/A_\Gamma$, choose an identification $\uc{S}/A_\Gamma\simeq \SS$ and consider the induced collapse in the quotient $c\colon X\to \SS$. Let $C\subseteq X$ be the unique vertex preimage, containing all vertices of $X$. Every hyperplane of $X$ that is not dual to an edge of $C$ has an orientation and label in $V$ pulled back from its image in $\SS$. For each hyperplane $H$ of $X$ dual to an edge of $C$, define a Whitehead partition $\mathbf{P}_H$ as follows:

    By Lemma~\ref{strongdownstairs}, $H$ induces a unique hyperplane of the $\cat$ subspace $C$. Orient the halfspaces of that hyperplane in $C$ arbitrarily as $H^+$, $H^-$. Note that these halfspaces partition the vertex set of $X$. Every element $a$ of $V^\pm$ corresponds to an oriented edge of $\SS$, hence an oriented hyperplane $H_a$ of $X$, not dual to an edge of $C$. Three cases arise:
    \begin{itemize}
        \item If every oriented edge dual to $H_a$ terminates in $H^+$, set $a\in P_H$.
        \item If every oriented edge dual to $H_a$ terminates in $H^-$, set $a\in P_H^*$.
        \item Otherwise, set $a\in L_H$.
    \end{itemize}
    This defines a partition $\mathbf{P}_H$ as $V^\pm = P_H\sqcup P_H^*\sqcup L_H$.

    First, we make the following observation: Let $e$ be any edge of $X$ starting in $H^-$ and terminating in $H^+$, with dual hyperplane $H_e$. We claim that every hyperplane transverse to $H_e$ is transverse to $H$. If $e$ is in $C$, $e$ must be dual to $H$ and the claim holds. If $e$ is not in $C$, let $p$ be a geodesic edge path in $C$ joining the endpoints of $e$. Some edge of $p$ must be dual to $H$, and the claim holds by Lemma~\ref{tube}. Note that in particular, $H$ is not transverse to $H_e$.

    This observation has several consequences. First, assume that $x\in P_H$, $x^{-1}\in P_H^*$, and $y\in lk(x)$. By definition of $\mathbf{P}_H$, any oriented edge dual to $H_x$ starts in $H^-$ and terminates in $H^+$. Moreover, $H_y$ is transverse to $H_x$ because their image hyperplanes in $\SS$ are transverse. Therefore, by the observation, $H_y$ is transverse to $H$, hence $y\in L_H$. This proves that $lk(x)\subseteq L_H$, the second assumption of the definition of a Whitehead partition.

    Next, note that $a\in L_H$ if and only if $H_a$ and $H$ are transverse, in particular $L_H$ is symmetric. Indeed, if $a\in L_H$, there exist two oriented edges $e^+,e^-$ dual to $H_a$ with $e^+$ terminating in $H^+$ and $e^-$ in $H^-$. Up to taking $e^+$ and $e^-$ maximally close in the carrier of $H_a$, they are opposite edges in a square. Let $f$ be the edge of that square joining the terminal endpoints of $e^+$ and $e^-$. Its dual hyperplane $H_f$ is transverse to $H_a$. By the observation, $H_a$ must be transverse to $H$. The reverse implication is clear.

    Furthermore, note that if some edge dual to $H_a$ has endpoints on opposite sides of $H$, $H_a$ and $H$ are not transverse by the observation, hence $a,a^{-1}\notin L_H$. Therefore, $a\in P_H$ and $a^{-1}\in P_H^*$ or vice versa. It is clear that conversely, if $a\in P_H$ and $a^{-1}\in P_H^*$ or vice versa, every edge dual to $H_a$ has endpoints on opposite sides of $H$.

    Now assume that $x\in P_H$, $y\in P_H^*$ are not inverses yet $x$ and $y$ commute. This means that their dual hyperplanes $H_x$ and $H_y$ are transverse. This transversality appears in a square of $X$. Some vertex of that square is both the terminal endpoint of an edge dual to the oriented $H_x$ and an edge dual to the oriented $H_y$. This vertex must then belong $H^+$ and $H^-$, a contradiction. Therefore, the third assumption of the definition of a Whitehead partition holds.

    We need to prove the existence of a basepoint with link $L_H$. By assumption, $X$ is spatial, therefore there exists a strong collapse $X\to S'$ where $S'$ is combinatorially isomorphic to $\SS$ and $H$ is not collapsed. Let $C'\subseteq X$ be the vertex preimage of this collapse. Pick $e$ any edge dual to $H$ and $p$ a geodesic path in $C'$ joining the endpoints of $e$. Some oriented edge of $p$ must start in $H^-$ and terminate in $H^+$. This edge is not in $C$, otherwise it would be dual to $H$, a contradiction because $H$ is not collapsed in $S'$. Therefore this edge is labelled by some $b\in V^{\pm}$. By Lemma~\ref{tube} applied to $e$ in $X\to S'$, every hyperplane transverse to $H$ is transverse to $H_b$. Conversely, by the observation in $X\to S$, every hyperplane transverse to $H_b$ is transverse to $H$. In particular, for every $a\in V^\pm$, $a\in L_H$ if and only if $H_a$ is transverse to $H_b$, if and only if $a\in lk(b)$. Since $b\notin lk(b)=L_H$, $b\in P_H$ and $b^{-1}\in P_H^*$ by construction. Therefore $b$ is a basepoint for $\mathbf{P}_H$.

    Now we make the following claim: the sides of $\mathbf{P}_H$ both have at least two elements, and for every hyperplane $K$ of $C$ distinct from $H$, $\mathbf{P}_K$ is a distinct partition from $\mathbf{P}_H$, even up to exchange of sides.

    The consequence of this claim is that the $\mathbf{P}_H$, for $H$ ranging across all hyperplanes of $C$, form a family of pairwise distinct Whitehead partitions. Assume for the sake of contradiction that the claim does not hold. Then either $(1.)$ some side of $\mathbf{P}_H$, say $P_H$ contains only the unique basepoint $b$ of $\mathbf{P}_H$, or $(2.)$ some hyperplane $K$ of $C$ satisfies, up to changing chosen sides, that $P_H = P_K$, $P^*_H = P^*_K$, and $K$ is in the halfspace $H^+$ of $C$. In the former case, let $K = H_b$ be the (two-sided) hyperplane corresponding to $b$ and orient it according to $b$ by choosing a halfspace $k$ of its carrier contained in $H^+$ (the other halfspace of the carrier is contained in $H^-$ because $b^{-1}\in P^*_H$). In the latter case, let $k$ denote the intersection of $K^-$ with the carrier of $K$, which is contained in $H^+$. The inclusion $k\subseteq H^+$ of convex subcomplexes of $C$ holds in both cases. Hence the separator in $C$, $\s(k\mid H^-)$ contains $H$. Let $H_0$ lie in this separator, maximally close from $k$, so that the carrier of $H_0$ intersects $k$. By definition, $H_0$ and $K$ are not transverse. We will prove that $\{H_0,K\}$ forms a redundant pair of hyperplanes, contradicting the fact that $X$ is spatial.

    First, we prove that $H$, $H_0$ and $K$ have the same set of transverse hyperplanes. It is clear that $H$ and $K$ have the same set of transverse hyperplanes labelled by generators ($L_H = lk(b)$ in Case $(1.)$ and $L_H = L_K$ in Case $(2.)$). Given a hyperplane $H_a$ in this set, both $k$ and the intersection of the carrier of $H$ with $H^-$ contain some edge labeled $a$, while being on opposite halfspaces of $H_0$ in $C$. By definition, $a\in L_{H_0}$, and $H_a$ and $H_0$ are transverse. Conversely, assume $H_a$ is a hyperplane transverse to $H_0$. This means that $a\in L_{H_0}$, i.e.~there exist oriented edges dual to $H_a$ terminating on either halfspace of $H_0$ in $C$. In Case $(1.)$, since one halfspace of $H_0$ is contained in $H^+$, this implies $a\in L_H\cup P_H = L_H\cup \{b\}$. Yet, $H_b = K$ is not transverse to $H_0$, thus $a\neq b$, and $a\in L_H$. This proves that $H_a$ is transverse to $H$ and $K$. In Case $(2.)$, one halfspace of $H_0$ in $C$ is still contained in $H^+$ and the other halfspace is contained in $K^-$. This implies, given $L_H = L_K$ that $a\in L_H\cup (P_H\cap P_K^*)$. Yet, $P_H\cap P_K^* = P_K\cap P_K^* = \emptyset$ by assumption. Thus $a\in L_H$ and $H_a$ is transverse to $H$ and $K$ once again. This proves that $H$, $H_0$ and $K$ all have the same set of transverse hyperplanes labelled by generators.
    
    Now, given a hyperplane $H'$ of $C$, let $b'$ be a basepoint for $\mathbf{P}_{H'}$, chosen as above so that $H'$ and $H_{b'}$ have the same transverse hyperplanes: by the previous argument, $H_{b'}$ is either transverse to all three $H$, $H_0$, and $K$, or to none of them. Thus, the same goes for $H'$, and $H$, $H_0$, $K$ all have the same transverse hyperplanes.

    Using this result and the facts that $H_0$ and $K$ do not self osculate and no pair of hyperplanes of $X$ inter-osculate, we obtain the following fact: the product $[0,2]\times k$ immerses in $X$ with $\{1\}\times k$ mapped to $k$ identically, and every edge corresponding to $[0,1]$ (resp. $[1,2]$) mapped to an edge dual to $H_0$ (resp. $K$). This is proved by a straightforward induction, building immersions of $[0,2]\times A$ for $A$ an increasingly large connected subcomplex of $k$, starting with $A$ equal to a single vertex in the intersection of $k$ and the carrier of $H_0$. Note that the image of $[0,1]\times k$ (resp. $[1,2]\times k$) under the immersion is exactly the carrier of $H_0$ (resp. $K$). 

    Finally, let $e$ be an edge of $X$ incident to $k$ and not contained in the image of the immersion previously defined. This means that the dual hyperplane $H_e$ is neither equal nor transverse to $K$ or $H_0$. Three cases arise:
    \begin{itemize}
        \item If $e$ is not an edge of $C$, then $e$ has a label $a\notin L_H$. The terminal endpoint of $e$ is in $k\subseteq H^+$, hence $a\in P_H$, by definition of the latter. Now in Case $(1.)$, $P_H = \{b\}$, hence $a=b$ and $e$ is dual to $H_b = K$ which contradicts the assumption. In Case $(2.)$, as before, $a\notin L_K$ and the terminal endpoint of $e$ is in $k\subseteq K^-$ hence $a\in P^*_K$. This contradicts the fact that $P_H = P_K$ and $P^*_H = P^*_K$ are disjoint. 
        \item If $e$ is an edge of $C$ and $H_e$ does not separate $H$ and $k$ in $C$, then one halfspace of $H_e$ contains both $H$ and $k$, and the other halfspace is contained in $H^+$ (and $K^-$ in Case $(2.)$). Up to exchanging halfspaces of $H_e$, say $H_e^+\subseteq H^+$. Letting $a$ be a basepoint of $\mathbf{P}_{H_e}$, oriented edges labeled $a$ all terminate in $H_e^+$. Then in Case $(1.)$, $P_H = \{a\}$, contradicting the assumption as before, and in Case $(2.)$, $a\in P_H\cap P^*_K$ also yielding the same contradiction.
        \item If $e$ is an edge of $C$ and $H_e$ separates $H$ and $k$ in $C$, then $H_e$ and $H_0$ must be transverse. Indeed, $H_e$ cannot separate $H_0$ and $k$, by definition of $H_0$, yet the carrier of $H_e$ intersects $k$. This contradicts the assumption that $H_e$ and $H_0$ are not transverse.
    \end{itemize}
    All three cases lead to a contradiction, proving that such an edge $e$ cannot exist. It is then clear that $H_0$ and $K$ are the images of some hyperplane in a subdivision, proving that $\{H_0,K\}$ is a redundant pair. Since $X$ is spatial, it has no redundant pairs: this is a contradiction, and the claim holds.

    \bigskip

    Next, let $H,K$ be distinct hyperplanes dual to edges of $C$. We claim that the Whitehead partitions $\mathbf{P}_H$ and $\mathbf{P}_K$ are compatible.
    
    Two cases arise:
    \begin{itemize}
        \item If at least one of $H^+\cap K^+$, $H^+\cap K^-$, $H^-\cap K^+$ or $H^-\cap K^-$ is empty, the corresponding intersection of sides $P_H\cap P_K$, $P_H\cap P_K^*$, $P_H^*\cap P_K$ or $P_H^*\cap P_K^*$ is empty. Since $\mathbf{P}_H$ and $\mathbf{P}_K$ are distinct even up to switching sides, all three other intersections are non-empty, and $\mathbf{P}_H$ and $\mathbf{P}_K$ are compatible.
        \item Otherwise, there exists a quadruple of points witnessing the transversality of $H$ and $K$ in $C$. Let $b,c$ be basepoints for $\mathbf{P}_H,\mathbf{P}_K$ respectively. Since $H$ is transverse to $K$, $H$ is transverse to $H_c$, thus $H_b$ is transverse to $H_c$. In particular $b\in lk(c)$, thus $\mathbf{P}_H$ and $\mathbf{P}_K$ are adjacent, hence compatible.
    \end{itemize}

    This proves that partitions in the family $\mathbf{\Pi}$ of all Whitehead partitions $\mathbf{P}_H$ for $H$ a hyperplane dual to an edge of $C$ are pairwise compatible.

    \medskip
    
    Our last claim is the existence of a combinatorial isomorphism $\SS^\mathbf{\Pi}\simeq X$ such that the composition $\SS^\mathbf{\Pi}\simeq X\to \SS$ is equal to the canonical collapse $c_\mathbf{\Pi}$. Indeed, write $\mathbf{\Pi} = \{\mathbf{P}_{H_1}\dots \mathbf{P}_{H_k}\}$, and $L_i = lk(\mathbf{P}_{H_i})$. A vertex $z$ of $\SS^\mathbf{\Pi}$ corresponds to a choice of sides $P_{H_i}^\times$ for each $\mathbf{P}_{H_i}$ such that for $i\neq j$, $P_{H_i}^\times$ and $P_{H_j}^\times$ intersect or $\mathbf{P}_{H_i}$ and $\mathbf{P}_{H_j}$ are adjacent. Such a choice of sides determines a choice of one halfspace $H_i^\times$ for each $i$. For $i\neq j$, if $a\in P_{H_i}^\times\cap P_{H_j}^\times$, any terminal vertex of an edge dual to $H_a$ is in $H_i^\times \cap H_j^\times$. Otherwise, $\mathbf{P}_{H_i}$ and $\mathbf{P}_{H_j}$ are adjacent, thus $H_i$ and $H_j$ are transverse and the halfspaces $H_i^\times$ and $H_j^\times$ intersect once again. By the Helly property in the $\cat$ cube complex $C$, the intersection of all chosen halfspaces is non-empty, yet it cannot contain more than one vertex since a halfspace has been chosen for each hyperplane of $C$. Therefore this intersection is a unique vertex, which we set as the image of $z\in \SS^\mathbf{\Pi}$. This yields a map of $0$-skeleta $(\SS^\mathbf{\Pi})^{(0)}\to X^{(0)}$. Conversely, a vertex $z$ of $X$ determines a choice of halfspaces $H_i^\times$ for all hyperplanes dual to edges of $C$, and a corresponding choice of sides $P_{H_i}^\times$ for all partitions of $\mathbf{\Pi}$. For every $i\neq j$, if $H_i$ and $H_j$ are transverse, $\mathbf{P}_{H_i}$ and $\mathbf{P}_{H_j}$ are adjacent. Otherwise, since $H_i^\times$ and $H_j^\times$ intersect, one is contained in the other, say $H_i^\times \subset H_j^\times$. Let $b$ be a basepoint for $H_i$ and set $a=b$ if the chosen halfspace is $H_i^+$, $a=b^{-1}$ if the chosen halfspace is $H_i^-$. By construction, every oriented edge dual to $H_a$ has terminal vertex in $H_i^\times\subset H_j^\times$, therefore $a$ is contained in the intersection of the corresponding sides $P_{H_i}^\times \cap P_{H_j}^\times$. This proves that this choice of sides defines a region, i.e.~a unique vertex of $\SS^\mathbf{\Pi}$, which is the inverse image of $z$. Thus, we constructed a bijection between the $0$-skeleta of $\SS^\mathbf{\Pi}$ and $X$.
    
    It is clear by construction of $\SS^\mathbf{\Pi}$ that two vertices of $\SS^\mathbf{\Pi}$ are joined by an edge labelled $\mathbf{P}_{H_i}$ if and only if their images are joined by an edge dual to $H_i$. This means that the bijection extends as a bijection of $1$-skeleta between the unique vertex preimage of $\SS^\mathbf{\Pi}$ and $C$. Assume two vertices $z_1,z_2$ of $\SS^\mathbf{\Pi}$ are joined by an edge labelled $a\in V^\pm$. This means that for each side $P_{H_i}^\times(z_1)$ chosen for $z_1$ (resp. $P_{H_i}^\times(z_2)$ chosen for $z_2$), $a^{-1}\in P_{H_i}^\times(z_1) \cup L_i$ (resp. $a\in P_{H_i}^\times(z_2) \cup L_i$), and the choice of sides of $\mathbf{P}_{H_i}$ is different for $z_1$ and $z_2$ if and only if $a\in single(\mathbf{P}_{H_i})$. Now let $x_1,x_2\in X$ be the images of $z_1,z_2$ under our bijection and consider any oriented edge $e$ in $X$ dual to $H_a$, with endpoints $y_1,y_2$. For every hyperplane $H_i$ dual to an edge of $C$, if $a\in double(P_{H_i})$ (resp. $double(P_{H_i}^*)$), that side of $\mathbf{P}_{H_i}$ is chosen in the regions both for $z_1$ and $z_2$. Thus, both $x_1$ and $x_2$ are in the halfspace $H_i^+$ (resp. $H_i^-$). By definition of $\mathbf{P}_{H_i}$, both $y_1$ and $y_2$ are in that halfspace as well: $H_i$ does not separate $x_1$ from $y_1$ nor $x_2$ from $y_2$ in $C$. If $a\in P_{H_i}$ and $a^{-1}\in P_{H_i}^*$, the region for $z_1$ has the side $P_{H_i}^*$ and the region for $z_2$ has the side $P_{H_i}$, thus $x_1\in H_i^-$ and $x_2\in H_i^+$. By definition of $\mathbf{P}_{H_i}$ again, $y_1\in H_i^-$ and $y_2\in H_i^+$, thus $H_i$ does not separate $x_1$ from $y_1$ nor $x_2$ from $y_2$ in $C$. The same conclusion holds if $a\in P_{H_i}^*$ and $a^{-1}\in P_{H_i}$. Finally, if $a\in L_i$, $a\notin single(\mathbf{P}_{H_i})$, therefore the same side of $\mathbf{P}_{H_i}$ is chosen for $z_1$ and $z_2$, thus $x_1$ and $x_2$ are on the same halfspace of $H_i$. Moreover, as seen earlier in the proof, $y_1$ and $y_2$ are on the same halfspace of $H_i$ as well. Therefore, $H_i$ separates $x_1$ and $y_1$ if and only if it separates $x_2$ and $y_2$. We proved that the separators in $C$ $\s(x_1\mid y_1)$ and $\s(x_2\mid y_2)$ are equal and contained in the set of hyperplanes $H_i$ with $a\in L_i$, i.e.~that are transverse to $H_a$. Let $p$ be a geodesic in $C$ joining $y_1$ to $x_1$. By specialness of $X$ (Lemma~\ref{spatialspecial}), every edge of $p$ inductively spans a square with an oriented edge dual to $H_a$. In particular, there exists an oriented edge dual to $H_a$ at $x_1$, with terminal endpoint $y_3$. Yet by construction, a hyperplane of $C$ separates $y_3$ from $x_2$ in $C$ if and only if it separates $y_1$ from $x_1$, if and only if it separates $y_2$ from $x_2$. Thus $y_3=y_2$ and there exists an oriented edge (unique by specialness) from $x_1$ to $x_2$ dual to $H_a$, which we set as the image of the edge from $z_1$ to $z_2$ labelled $a$. This provides a map of $1$-skeleta $(\SS^\mathbf{\Pi})^{(1)}\to X^{(1)}$. Its inverse is constructed symmetrically as before, using specialness of $\SS^\mathbf{\Pi}$. This provides an isomorphism of $1$-skeleta respecting the hyperplanes and their labellings.

    In a special cube complex, a length $4$ edge cycle is the attaching cycle of a (unique) square if and only if its opposite edges are in the same hyperplanes, with the correct orientation, and the two hyperplanes are transverse. However, hyperplanes labelled $\mathbf{P}_{H_i}$ and $\mathbf{P}_{H_j}$ in $\SS^\mathbf{\Pi}$ are transverse if and only if $\mathbf{P}_{H_i}$ and $\mathbf{P}_{H_j}$ are adjacent, if and only if $H_i$ and $H_j$ are transverse in $X$. Moreover, hyperplanes labelled $\mathbf{P}_{H_i}$ and $a\in V^\pm$ are transverse in $\SS^\mathbf{\Pi}$ if and only if $a\in L_i$, if and only if $H_a$ is transverse to $H_i$ in $X$. Finally, hyperplanes labelled $v,w\in V$ are transverse in $\SS^\mathbf{\Pi}$ if and only if $v$ and $w$ commute in $A_\Gamma$, if and only if $H_v$ and $H_w$ project to transverse hyperplanes of $\SS$, if and only if $H_v$ and $H_w$ are transverse. Therefore, our isomorphism extends to the $2$-skeleta of $\SS^\mathbf{\Pi}$ and $X$. Since both are locally $\cat$ cube complexes, the link condition allows to extend the isomorphism to the full cube complexes $\SS^\mathbf{\Pi}\to X$, so that $\SS^\mathbf{\Pi}\to X\to \SS$ is the canonical collapse.
    
    Finally, by definition, any cospatial action admits at least one strong collapse to a Salvetti-like action. The previous argument proves then that every cospatial action is blow-up-like.

\end{proof}

Theorem~\ref{treelikecompatible}, Lemmas~\ref{blowupstrong} and \ref{weakblowupgood}, and Proposition~\ref{spatialblowup} yield the following corollaries.

\begin{cor}
\label{strong}
    Any equivariant collapse between cospatial actions of $A_\Gamma$ is strong.
\end{cor}
\begin{cor}
\label{subtreelike}
    Any marking-preserving collapse between spatial cube complexes is isomorphic to the collapse of a family of hyperplanes labelled by partitions in some blow-up $\SS^\mathbf{\Pi}$ of $\SS$.
\end{cor}

\begin{cor}
\label{renamepartition}
    Let $\overline{\H}$ be a tree-like family of hyperplanes in a blow-up $\SS^\mathbf{\Pi}$. There exists a family $\mathbf{\Pi}'$ of pairwise compatible Whitehead partitions and relabeling of the hyperplanes giving an isomorphism $\SS^\mathbf{\Pi}\to \SS^{\mathbf{\Pi}'}$ mapping $\overline{\H}$ to $\mathbf{\Pi}'$.
\end{cor}
The latter corollary was used implicitly in \cite{twistedos}, Definition~2.16.

\bigskip

In view of Remark~\ref{collapseredundant}, it could be possible in general to have a locally $\cat$ cube complex $X$ marked by $A_\Gamma$ with two distinct marking-preserving strong collapses with ranges isomorphic as marked complexes. However, the following lemma proves that it is not the case for spatial cube complexes.
\begin{lm}
\label{nobigons}
    Let $(X,m)$ be a marked spatial cube complex. Let $c_1\colon (X,m)\to (Y_1,m_1)$, $c_2\colon (X,m)\to (Y_2,m_2)$ be marking-preserving collapse maps with spatial range, collapsing distinct families of hyperplanes of $X$. Then, there is no marking-preserving isomorphism $(Y_1,m_1)\to (Y_2,m_2)$.
\end{lm}
\begin{proof}
Assume for the sake of contradiction that there exists a marking-preserving isomorphism $\iota \colon (Y_2,m_2)\simeq (Y_1,m_1)$. By Corollary~\ref{subtreelike}, there exists an isomorphism between $X$ and some blow-up $\SS^\mathbf{\Pi}$ of $\SS$ such that all hyperplanes collapsed by $c_1$ are labelled by partitions. Let $d\colon (Y_1,m_1)\to (\SS,m_3)$ be a marking preserving collapse corresponding to the collapse of all remaining hyperplanes labelled by partitions. Up to replacing $c_1$ by $d\circ c_1$ and $c_2$ by $d\circ \iota \circ c_2$, assume $Y_1 = \SS$. Thus, $c_1\colon \SS^\mathbf{\Pi}\to \SS$ is the standard collapse. Up to replacing $X$ with the range of the collapse of all hyperplanes collapsed by $c_2$ labelled by partitions (with a marking making all collapses involved marking-preserving), assume $c_2$ only collapses hyperplanes labelled by generators. The set $\H$ of hyperplanes collapsed by $c_2$ forms a tree-like family by Theorem~\ref{treelikecompatible}. By assumption, $\iota_*\circ (c_2)_*\circ (c_1)_*^{-1}$ is an inner automorphism of $\pi_1Y_1$. In particular, for every edge cycle $\gamma$ in $X$, shortest edge cycles in the homotopy classes of $c_1\circ \gamma$ and $c_2\circ \gamma$ have the same length. Note that by Lemma~\ref{collapse1}, $(2.)$, if $\gamma$ is shortest in its homotopy class, so are $c_1\circ \gamma$ and $c_2\circ \gamma$.  Let $C\subset \SS^\mathbf{\Pi}$ be the $\cat$ subcomplex given by the only vertex preimage under $c_1$.

Let $\mathbf{P}\in \mathbf{\Pi}$ and let $H$ be the corresponding hyperplane of $C$. Let $b\in single(\mathbf{P})$, and let $\gamma$ be a shortest edge cycle in $\SS^\mathbf{\Pi}$ projecting via $c_1$ to the edge loop of $Y_1 = \SS$ labelled by $b$: $\gamma$ has one edge labelled $b$ and all of its remaining edges are labelled by partitions. Since $c_2\circ \gamma$ must have only one edge, $|\gamma| = 2$. Hence, $c_2$ must collapse the edge labelled $b$, and the endpoints of the edge labelled $b$ in $\gamma$ are only separated by $H$ in $C$. By construction of $\SS^\mathbf{\Pi}$, this means that $b$ is not split by any partition in $\mathbf{\Pi}$ other than $\mathbf{P}$.

Now assume $single(\mathbf{P})$ contains $b'\notin \{b,b^{-1}\}$, with, say $b,b'\in P$. Let $\gamma'$ be a shortest edge cycle in $\SS^\mathbf{\Pi}$ projecting via $c_1$ to the length $2$ edge cycle of $Y_1 = \SS$ labelled by $bb'$: $\gamma'$ has one edge $e$ labelled $b$, one edge $e'$ labelled $b'$, and all of its remaining edges are labelled by partitions. Write $\gamma' = epe'p'$, with $p,p'$ only labelled by partitions. Since $c_2\circ \gamma'$ must be of length $2$, $|p| + |p'| = 2$. Since the terminal endpoint of $e$ and the initial endpoint of $e'$ are on opposite halfspaces of $H$ in $C$, $p$ has at least one edge labelled $\mathbf{P}$. Symmetrically, $p'$ has at least one edge labelled $\mathbf{P}$. Hence both $p$ and $p'$ are single edges labelled $\mathbf{P}$. However, by the argument above, $b$ and $b'$ are not split by any partition other than $\mathbf{P}$. Thus, the two endpoints of $e$ are separated only by $H$ in $C$, and the same holds for the two endpoints of $e'$. Therefore, by Lemma~\ref{square}, $p$ joins the endpoints of $e$ and $p'$ joins the endpoints of $e'$. Hence, $e^{-1}e'$ forms a closed length $2$ path in $\SS^\mathbf{\Pi}$, projecting via $c_1$ to a path labelled $b^{-1}b'$ in $\SS$, and via $c_2$ to the trivial path. This contradicts the fact that $c_1$ and $c_2$ are homotopy equivalences. Thus, $single(\mathbf{P})$ contains only two elements: the unique basepoint $b$ of $\mathbf{P}$ and its inverse. This means that the hyperplanes of $\mathbf{\Pi}$ have pairwise distinct unique basepoints. Since $|\H| = |\mathbf{\Pi}|$, $\H$ is exactly the family of hyperplanes labelled by such basepoints.

Then, since $P$ and $P^*$ each have at least two elements, choose any $a\in P\setminus\{b\}$, $c\in P^*\setminus\{b^{-1}\}$. Since $a,c\notin single(\mathbf{P})$, $a\in double(P)$ and $c\in double(P^*)$. Let $\gamma''$ be a shortest edge cycle in $\SS^\mathbf{\Pi}$ projecting via $c_1$ to the length $2$ edge cycle $ac$. By similar arguments as before, $\gamma''$ decomposes as $epe'p'$, where $e$ is labelled $a$, $e'$ is labelled $c$, and $p,p'$ are paths only labelled by partitions, each having one edge labelled $\mathbf{P}$. Since $c_2\circ \gamma''$ has length $2$, $p$ and $p'$ are single edges labelled $\mathbf{P}$ and $a,c$ are collapsed by $c_2$, i.e.~are basepoints of partitions $\mathbf{Q},\mathbf{Q}'\in \mathbf{\Pi}$. These partitions are distinct because they have distinct unique basepoints. As before, the endpoints of $e$ (resp. $e'$) are only separated in $C$ by the hyperplane dual to $\mathbf{Q}$ (resp. $\mathbf{Q}'$). Hence, these endpoints are joined by some edge $e_\mathbf{Q}$ (resp. $e_{\mathbf{Q}'}$) labelled $\mathbf{Q}$ (resp. $\mathbf{Q}'$). The path of length $4$ $e_\mathbf{Q}pe_{\mathbf{Q}'}p'$ is closed in $C$, hence bounds a square by Lemma~\ref{square}. This is a contradiction because the hyperplanes dual to $e_\mathbf{Q}$ and $e_{\mathbf{Q}'}$ are distinct, proving the result.
\end{proof}

\begin{defi}
\label{defuntwisted}
    Let $A_\Gamma$ act freely and cocompactly on a $\cat$ cube complex $Z$ by combinatorial isometries. Consider also the standard action of $A_\Gamma$ on the universal cover $\uc{\SS}$. The action of $A_\Gamma$ is \emph{untwisted}, or \emph{coarse-median-preserving}, when there exists vertices $x\in \uc{\SS}$, $z\in Z$ and a constant $N$ such that for $a,b,c\in A$, the unique $g\in A$ such that $gx = \mu_{\uc{\SS}}(ax,bx,cx)$ satisfies $d_Z(gz, \mu_Z(az,bz,cz))\leq N$. This is equivalent to asking that both actions of $A_\Gamma$ on $Z$ and $\uc{\SS}$ induce the same coarse median structure on $A_\Gamma$.
\end{defi}

\begin{defi}
\label{defcat}
    Let $\mathcal{C}$ be the category whose objects are cospatial untwisted actions of $A_\Gamma$ and whose arrows are given by equivariant collapse maps, or, equivalently, given by strong collapse maps, thanks to Corollary~\ref{strong}.

    Let $\overline{\mathcal{C}}$ be the category of cospatial untwisted actions of $A_\Gamma$ and equivariant collapse maps, considered modulo $A_\Gamma$-equivariant isomorphism (the \emph{skeleton} of $\mathcal{C}$ in the categorical sense). Equivalently, it is the category of spatial cube complexes with coarse-median-preserving markings, and marking-preserving hyperplane collapses, considered modulo marking-preserving isomorphism. By Proposition~\ref{spatialblowup}, Lemma~\ref{nobigons} and the finiteness of the set of blow-ups of $\SS$, $\overline{\mathcal{C}}$ has countably many objects, finitely many arrows with fixed domain (resp. range), and at most one arrow with fixed domain and range.

    Let $\mathcal{N}$ be the \emph{nerve} of $\overline{\mathcal{C}}$ i.e.~the simplicial complex having one $0$-cell per object of $\overline{\mathcal{C}}$, one $1$-cell per non-identity arrow of $\overline{\mathcal{C}}$, and more generally one $k$-simplex per sequence of $k$ composable non-identity arrows of $\overline{\mathcal{C}}$. Equivalently, it is the simplicial complex of the poset of cospatial actions of $A_\Gamma$ (up to $A_\Gamma$-equivariant combinatorial isomorphism) ordered by equivariant collapse maps.
\end{defi}

\begin{thm}
\label{newspineth}
    The group $\unt(A_\Gamma)$ acts on untwisted cospatial actions by precomposition. It acts on $\mathcal{N}$ by combinatorial isomorphisms.

    There exists an $\unt(A_\Gamma)$-equivariant, order-preserving, combinatorial isomorphism between the spine of untwisted outer space $K_\Gamma$ and $\mathcal{N}$.
\end{thm}
In particular, using the results of \cite{untwistedos}, $\mathcal{N}$ is finite-dimensional and contractible, and the action of $\unt(A_\Gamma)$ on $\mathcal{N}$ is properly discontinuous and cocompact.

\begin{proof}
    Let $A_\Gamma$ act on $Z$ cospatially, and consider a strong collapse $c\colon Z\to S$ to a Salvetti-like action. Since $Z$ has finitely many orbits of hyperplanes, $c$ is quasi-isometric and preserves medians (Lemmas~\ref{collapse1} and \ref{quasiisometry}). Thus, the action on $S$ is untwisted if and only if the action on $Z$ is.
    
    Moreover, fixing a combinatorial isomorphism $\iota \colon S/A_\Gamma \to \SS$ and a vertex $x\in S$ induces a map $A_\Gamma \to A_\Gamma\cdot x \to \pi_1(S/A_\Gamma) \to \pi_1\SS = A_\Gamma$ which is an automorphism $\varphi$ of $A_\Gamma$. The action of $A_\Gamma$ on $S$ is then equivariantly isomorphic to the standard action of $A_\Gamma$ on $\SS$ precomposed with $\varphi$. By Theorem~\ref{CMP}, the action on $S$ is untwisted if and only if $[\varphi]\in \unt(A_\Gamma)$.

    Then, letting $[\psi]\in \unt(A_\Gamma)$, precomposing the action on $S$ by $\psi$ is equivariantly isomorphic to precomposing the action on $\SS$ by $\psi\circ\varphi$. Finally, the actions on $Z$ and $S$ precomposed by $\psi$ still differ by the equivariant strong collapse $c$. This proves that the action on $Z$ precomposed by $\psi$ is untwisted if and only if the original action on $Z$ was untwisted. This action is clearly cospatial and does not depend on the representative of $[\psi]$, up to equivariant automorphism of $Z$. Therefore, $\unt(A_\Gamma)$ acts on the vertices of $\mathcal{N}$ and the action is adjacency-preserving, i.e.~extends to an action on $\mathcal{N}$ by combinatorial isomorphisms.
    
    By Lemma~\ref{blowupstrong}, Lemma~\ref{weakblowupgood}, Proposition~\ref{spatialblowup}, and Remark~\ref{blowupspatial}, an action of $A_\Gamma$ is blow-up-like if and only if it is cospatial, and a map between cospatial actions is an equivariant collapse if and only if it lifts the collapse of a subfamily of a tree-like family of hyperplanes in a blow-up of $\SS$. Finally, it is clear that untwisted markings on the Salvetti complex in the sense of \cite{untwistedos} correspond to precompositions of the action on $\uc{\SS}$ by elements of $\unt(A_\Gamma)$, i.e.~untwisted Salvetti-like actions. This proves that the poset defining $K_\Gamma$ is $\unt(A_\Gamma)$-equivariantly isomorphic to the poset defining $\mathcal{N}$.
\end{proof}

\bibliographystyle{plain}
\bibliography{Spatial_cube_complexes}

\end{document}